\documentclass[12pt,twoside,reqno]{amsart}
\usepackage[colorlinks=true,citecolor=blue]{hyperref}
\usepackage{mathptmx, amsmath, amssymb, amsfonts, amsthm, mathptmx, enumerate, color}
\usepackage{graphicx}
\usepackage{subcaption}
\usepackage{comment}
\usepackage[utf8]{inputenc}
\usepackage{siunitx,booktabs}
\usepackage{arydshln}
\usepackage{setspace,mathtools,bm,algorithm,algpseudocode,
longtable,multirow,geometry,mathrsfs,tablefootnote,threeparttable}
\usepackage{xcolor}
\usepackage{algorithm,algpseudocode}

\usepackage{graphicx}
\usepackage{epstopdf}

\newtheorem{theorem}{Theorem}[section]

\newtheorem{lemma}{Lemma}[section]
\newtheorem{proposition}{Proposition}[section]
\theoremstyle{definition}
\newtheorem{definition}{Definition}[section]

\newtheorem{remark}{Remark}[section]

\newcounter{assumption}
\renewcommand{\theassumption}{\Alph{assumption}}

\makeatletter
\newenvironment{assumption}[1][]{%
  \refstepcounter{assumption}
  \par\medskip\noindent
  \textbf{Assumption \theassumption.}%
  \@ifnotempty{#1}{ (#1)}%
  \quad
}{\medskip}
\makeatother

\definecolor{dkcolor}{rgb}{0.8,0.0,0.0} 

\numberwithin{equation}{section}

\begin{document}
\setcounter{page}{1}

\vspace*{1.0cm}
\title[Inexact Zeroth-Order Nonsmooth and Nonconvex Stochastic Composite Optimization]
{Inexact Zeroth-Order Nonsmooth and Nonconvex Stochastic Composite Optimization and Applications}
\author[S. Pougkakiotis and D. Kalogerias]{Spyridon Pougkakiotis$^{1,*}$, Dionysis Kalogerias$^2$}
\maketitle
\vspace*{-0.6cm}

\begin{center}
{\footnotesize {\it

$^1$Department of Mathematics, King's College London, London, England, UK\\
$^2$Department of Electrical and Computer Engineering, Yale University, New Haven, CT, USA

}}\end{center}

\vskip 4mm {\small\noindent {\bf Abstract.}
In this paper we present an inexact zeroth-order method suitable for the solution of nonsmooth and nonconvex stochastic composite optimization problems, in which the objective is split into a real-valued Lipschitz continuous stochastic function and an extended-valued (deterministic) proper, closed, and convex one. The algorithm operates under inexact oracles providing noisy (and biased) stochastic evaluations of the underlying finite-valued part of the objective function. We show that the proposed method converges (non-asymptotically), under very mild assumptions, close to a stationary point of an appropriate surrogate problem which is related (in a precise mathematical sense) to the original one. This, in turn, provides a new notion of approximate stationarity suitable nonsmooth and nonconvex stochastic composite optimization, generalizing conditions used in the available literature.
\par In light of the generic oracle properties under which the algorithm operates, we showcase the applicability of the approach in a wide range of problems including large classes of two-stage nonconvex stochastic optimization and nonconvex-nonconcave minimax stochastic optimization instances, without requiring convexity of the lower level problems, or even uniqueness of the associated lower level solution maps. We showcase how the developed theory can be applied in each of these cases under general assumptions, providing algorithmic methodologies that go beyond the current state-of-the-art appearing in each respective literature, enabling the solution of problems that are out of reach for currently available methodologies.\\

 }

\renewcommand{\thefootnote}{}
\footnotetext{ $^*$Corresponding author.
\par
E-mail addresses: spyridon.pougkakiotis@kcl.ac.uk (S. Pougkakiotis), dionysis.kalogerias@yale.edu (D. Kalogerias).
\par
Received August --, 2025; Accepted --. }

\section{Introduction}

\par  Let $(\Omega,\mathscr{F},\mu)$ be a complete base probability space, and consider a random vector $\xi \colon \Omega \rightarrow \Xi \subset \mathbb{R}^d$, and its induced Borel space $(\Xi,\mathscr{B}(\Xi),P)$, where $P \colon \mathscr{B}(\Xi) \rightarrow [0,1]$ is the induced Borel measure. In this paper, we consider the following \emph{nonsmooth, nonconvex and stochastic composite optimization} problem:
\begin{equation} \label{eqn: main problem} \tag{P}
\min_{x \in \mathbb{R}^n} \phi(x) \triangleq \underbrace{\mathbb{E}\{F(x,\xi)\}}_{ \triangleq f(x)} + r(x),
\end{equation}
\noindent where $r \colon \mathbb{R}^n \rightarrow \overline{\mathbb{R}}$ is a closed, proper, convex and \emph{proximable} function (i.e., a function the proximity operator of which can be computed expeditiously). Throughout this work, we will make use of the following blanket assumption on \eqref{eqn: main problem}.
\begin{assumption}\label{assum: basic assumption} 
\noindent The following conditions are in effect for \textnormal{\eqref{eqn: main problem}}:
\begin{itemize}
\item[\textbf{(A1)}] The function $F(x,\cdot)$ is Borel measurable for all $x \in \mathbb{R}^n$. For a.e. $\xi \in \Xi$, the function $F(\cdot,\xi) \colon \mathbb{R}^n \rightarrow \mathbb{R}$ is $L(\xi)-$Lipschitz continuous with $\mathbb{E}\{L^2(\xi)\} \leq G^2$, for some $G > 0$. Moreover, we have that $f(x) = \mathbb{E}\{F(x,\xi)\}$, for all $x \in \mathbb{R}^n$;
\item[\textbf{(A2)}] We can draw i.i.d. samples from the law of $\xi$;
\item[\textbf{(A3)}] We have that $r$ is closed, proper, convex, and proximable (as already mentioned).
\end{itemize}
\end{assumption}
\subsection{Prior work and core contributions}
\par Nonsmooth and nonconvex stochastic optimization has received a lot of attention in recent years, due to its ubiquitous presence in machine learning and artificial intelligence applications. A seminal work on tackling such problems was originally proposed in \cite{pmlr:Zhang_etal}, in which the authors provided an interpolated normalized gradient method for the solution of \eqref{eqn: main problem} (assuming that $r = 0$ and that $F$ belongs to an appropriate sub-class of Lipschitz functions), and showed that it converges non-asymptotically towards a $(\delta,\epsilon)-$Goldestein stationary point (for additional details on this mode of convergence, see Definition \ref{def: Goldstein subdifferential} and Section \ref{sec: termination criteria}). This work built upon earlier developments due to Goldstein (see \cite{Goldstein}) and led to a series of works extending these results. Indeed, an improved interpolated normalized gradient variant was later proposed in \cite{NEURIPS2022_2c8d9636}, showing that its non-asymptotic convergence towards a $(\delta,\epsilon)-$Goldstein stationary point holds for any Lipschitz continuous function $F$. 
\par An alternative line of work, highly related to this paper, deviated from interpolated normalized gradient schemes, considering instead (randomized) zeroth-order stochastic optimization methods (see \cite{nesterov2017random,duchi2012randomized} and the references therein for an overview on randomized zeroth-order stochastic optimization). Indeed, as was originally identified in \cite{NEURIPS2022_Linetal}, the gradients of uniform randomized smoothed surrogates associated to \eqref{eqn: main problem} (again, assuming that $r = 0$) have a close (and mathematically precise) relation to the $(\delta,\epsilon)-$Goldstein subdifferential. In turn, they were able to show that the associated zeroth-order stochastic gradient schemes arising from such smoothing strategies also converge (non-asymptotically) in the Goldstein sense, much like interpolated normalized gradient schemes, albeit with a rate that depends on the problem dimension. Improved variants (in the sense of dimension-dependence) of the algorithm presented in \cite{NEURIPS2022_Linetal} were later proposed in \cite{Chen_etal_ICML,Cutkosky_ICML} and then in \cite{JMLR:KornowskiShamir}. Considerations about the inherent need for randomized smoothing were also discussed in \cite{pmlr-v195-jordan23a}.
\par Most works on nonsmooth and nonconvex optimization currently available in the literature focus on the unconstrained (non-composite) case (i.e., in the case where $r = 0$ in \eqref{eqn: main problem}). To the best of our knowledge, the case where $r \neq 0$ (and is allowed to be extended-valued) is only considered in \cite{Liu_etal_ICML} (although structured constrained formulations of \eqref{eqn: main problem} have been considered in other studies such as in \cite{GrimmerOptLetters}), where the authors propose a zeroth-order optimization scheme in the case where $r$ is an indicator to a closed convex and compact set. The authors in \cite{Liu_etal_ICML} generalize $(\delta,\epsilon)-$Goldstein stationarity to fit the their problem (cf. \cite[Definition 4.2]{Liu_etal_ICML}) by appropriately extending the well-known \emph{gradient mapping} (see \cite[Section 2.2.4]{nesterov2018lectures} for a definition of this mapping) using the Goldstein subdifferential; their proposed stationarity condition is different from the generalization proposed in this work (cf. Section \ref{sec: termination criteria}), which combines surrogate stationarity based on both the Moreau envelope (akin to that considered in \cite{DavisDrus_SIOPT} for weakly convex composite optimization) and the Goldstein subdifferential (as is done in standard unconstrained nonsmooth and nonconvex optimization; e.g., see \cite{pmlr:Zhang_etal}). The proposed notion of stationarity has several benefits compared to that considered in \cite{Liu_etal_ICML}. On the one hand, it readily enables the use of a generic convex regularizer $r$. On the other hand, it provides a natural framework for analyzing the (non-asymptotic) convergence of zeroth-order optimization schemes applied to \eqref{eqn: main problem}, while maintaining crucial connections to the Goldstein (approximate) stationarity in the unconstrained case (i.e., when $r = 0$).
\par More crucially, current algorithms studied in the literature, suitable for the solution of nonsmooth and nonconvex optimization problems, operate under the assumption of the availability of \emph{exact oracles} able to evaluate the stochastic function $F(x,\xi)$, for any $x \in \mathbb{R}^n$ and a.e. $\xi \in \Xi$. As will become clear in Sections \ref{subsec: applications} and \ref{sec: applications}, enabling the presence of inexactness in the evaluations of $F(x,\xi)$ is of paramount importance for several applications of practical interest. Thus, this work aims at closing core gaps in the current literature of nonsmooth and nonconvex stochastic optimization, by providing a natural condition of stationarity suitable for the constrained (or composite) case, while at the same time allowing for errors in the underlying stochastic function evaluations. Furthermore, by specializing the notion of an inexact oracle in the context of zeroth-order stochastic optimization, we provide new and general conditions on the associated oracle errors. In turn, this enables us to derive improved (non-asymptotic) convergence rate bounds under very reasonable oracle error conditions, by simply utilizing the properties of zeroth-order  optimization schemes.
\subsection{Related applications} \label{subsec: applications}
\par Problem \eqref{eqn: main problem} is prominent in a plethora of applications of great interest, stemming from machine learning to operational research and signal processing. Specifically, nonsmooth and nonconvex optimization involving Lipschitz continuous functions has received a lot of attention in the recent literature (e.g., see \cite{pmlr:Zhang_etal,NEURIPS2022_2c8d9636,NEURIPS2022_Linetal} and the references therein) due to its direct application on the training of neural networks which, when seen as compositional functions, often fail to satisfy standard assumptions like weak convexity, Lipschitz smoothness or even subdifferential regularity. Indeed, as is already mentioned in \cite{NEURIPS2022_2c8d9636}, most (sub)gradient-based methods rely on some form of subdifferential regularity, which fails when the function that is being optimized exhibits some ``downward cusps" (e.g., see the example $f(x) = (1-\max\{x,0\})^2$, given in \cite{NEURIPS2022_2c8d9636}).
\par As we have already hinted earlier, one major gap in the current literature of nonsmooth and nonconvex optimization is the derivation of algorithms that are able to operate under noisy and inexact function evaluations. This is especially important in cases where the function $F(\cdot,\xi)$ appearing in \eqref{eqn: main problem} is itself a (possibly nonconvex) optimization problem. In this case, under fairly general conditions (e.g., see the discussion in Section \ref{sec: applications} as well as the comprehensive exposition given in \cite{ShapiroPerturbationAnalysis}), one may be able to show that $F(\cdot,\xi)$ is (possibly Lipschitz) continuous, but not necessarily differentiable or even subdifferentially regular. In this regime, the assumption that $F(\cdot,\xi)$ can be evaluated exactly, for a.e. $\xi \in \Xi$, is quite strong (since its evaluation typically occurs via the utilization of an ``inner-layer" numerical optimization scheme). Two very important classes of problems that exhibit this behavior are (possibly nonconvex) two-stage stochastic programs and stochastic minimax optimization instances. Additionally, the same considerations apply in the context of hyperparameter tuning of black-box systems, the evaluation of which might be noisy and inexact (e.g., in case the objective function is evaluated via the utilization of a simulation process; the reader is referred to \cite[Section 4.2]{Pougk_SISC} for an example of hyperparameter tuning in this context).
\par More concretely, in the case of two-stage stochastic programming, the function $F(x,\xi)$ is defined as $F(x,\xi) = \min_{y \in \mathcal{Y}(x,\xi)} \hat{F}(x,y,\xi)$, where $\mathcal{Y}(x,\xi) \subset \mathbb{R}^m$ is the feasible set of the second-stage variable $y$. Two-stage stochastic programming problems appear in a plethora of applications in operational research and engineering. While many such instances are posed in the context of convex stochastic optimization (e.g., see the detailed exposition in \cite[Chapter 2]{ShapiroLecturesStochProg3rd}), nonconvex formulations are also highly relevant. A typical example arises in the context of beamforming optimization for wireless communication systems, in cases where the performance of the underlying network can be improved by tuning an appropriate set of parameters in a long timescale, jointly with optimizing short-time scale (i.e., recourse) variables \cite{2S_Liu2021,2S_Zhai2022, 2S_Zhao2022, 2S_Zhao2024}; see also the recent line of work \cite{Hashmi_etal_ICASSP,Hashmi_etal_IEEETran,Pougk_etal_IRSinexact} within the more specialized but challenging context of intelligent reflecting surface-assisted beamforming. Another separate example of an application of two-stage stochastic optimization on certain meta-learning problems arising in the area of computer vision may be found in \cite{CV_10203333}.
\par In the case of minimax stochastic optimization, we may separate two distinct cases. The first class of instances arises by letting $F(x,\xi) = \max_{y \in \mathcal{Y}(x,\xi)} \hat{F}(x,y,\xi)$, where $\mathcal{Y}(x,\xi) \subset \mathbb{R}^m$ is the feasible set of the \textit{adversarial variable} $y$. In essence, in this formulation, the adversary is given access to instantaneous information and thus \textit{from this point of view} the underlying stochastic minimax optimization problem is fairly similar to two-stage stochastic programming models, although structurally different. One of the most important applications of this problem formulation arises in the context of building neural networks robust to adversarial attacks (e.g., see the seminal paper \cite{madry2018towards} and numerous follow-up works). Problems of this form are practically solved via approximate stochastic hypergradient descent-type schemes (again, see \cite{madry2018towards}), although without any theoretical guarantees, despite the inherent assumption that the feasible set $\mathcal{Y}$ of the adversarial variable is independent of both $x$ and $\xi$. Nonetheless, we conjecture that the stochastic hypergradient descent approach proposed in \cite{Pougk_etal_IRSinexact} in the context of nonconvex two-stage stochastic programming can possibly be adapted in this case and be shown to be non-asymptotically convergent under certain regularity and structural assumptions.
\par The second class of stochastic minimax optimization instances, which is very well-studied in the literature, arises by assuming that the adversary only has access to ergodic information, in which case the objective function of \eqref{eqn: main problem} reads $F(x,\xi) = \min_{y \in \mathcal{Y}(x)} \mathbb{E}\{\hat{F}(x,y,\xi)\}$, where once again $\mathcal{Y}(x) \subset \mathbb{R}^m$ is the feasible set of the adversarial variable $y$. Such problems have multiple applications, especially in the context of machine learning, including generative adversarial networks (e.g., \cite{goodfellow2014generative}), online adversarial learning (e.g., \cite{Cesa-Bianchi_Lugosi_2006}), robust training of neural networks (e.g., \cite{awasthi2021certifying}), nested optimization in reinforcement learning (e.g., see \cite{10.5555/2888116.2888133}), and distributionally robust optimization (e.g., see \cite{Rahimian2019DistributionallyRO}), to name a few. In most of these applications, it is assumed that $\mathcal{Y}$ does not depend on $x$, and under assumptions like Lipschitz smoothness on $\hat{F}$ and lower level concavity, typical solution methods rely either on stochastic gradient descent-ascent variants (e.g., \cite{JMLR:v26:22-0863}) or the extragradient method under additional conditions (e.g., \cite{pmlr-v238-emmanouilidis24a}), although recent works have also investigated the fully nonconvex-nonconcave setting under alternative (and even stronger) structural assumptions (e.g., see \cite{pmlr-v195-daskalakis23b,diakonikolas2021efficient,grimmer2022landscape}).
\par In this work, we showcase that a single algorithmic strategy, as proposed in this work, can be readily adapted and applied to each of these problems classes, resulting in solution methods that operate under very general assumptions, going beyond the current state-of-the-art in each respective literature, albeit at the cost of two function evaluations at each iteration (i.e., two inexact inner-problem solutions at adjacent outer-problem points). Nonetheless, we showcase that despite the added computational overhead, the proposed methodology is otherwise very efficient and can operate in regimes that are inaccessible to alternative approaches, offering strong modeling capabilities and robustness.
\par We now provide an overview of this paper. Specifically, in Section \ref{subsec: notation} we summarize the notation used throughout this work. Then, in Section \ref{sec: preliminaries} we provide necessary background material on nonsmooth optimization, variational analysis, randomized smoothing and evaluation oracles. Subsequently, in Section \ref{sec: termination criteria}, we propose a new notion of approximation stationarity that is suitable for nonsmooth and nonconvex stochastic composite optimization. This is then used in Section \ref{sec: inexact zeroth-order method}, where we derive the proposed algorithm and show its non-asymptotic convergence under minimal assumptions. The results of Section \ref{sec: inexact zeroth-order method} are then specialized to fit different applications in Section \ref{sec: applications} to showcase the power and generality of the proposed methodological framework. Finally, we close this paper by collecting some conclusions in Section \ref{sec: conclusions}.
\subsection{Notation} \label{subsec: notation}  Throughout this work we write $\|\cdot\|$ to denote the standard Euclidean norm. Given some positive constant $\epsilon > 0$ and some $x \in \mathbb{R}^n$, we let $\mathbb{B}_{\epsilon}(x)$ denote the open $\epsilon-$ball around $x$ on $\mathbb{R}^n$, i.e., $\mathbb{B}_{\epsilon}(x) \triangleq \{y \in \mathbb{R}^n\ \vert\ \|y-x\| < \epsilon\}$. Similarly, the closed $\epsilon-$ball around $x$ on $\mathbb{R}^n$ is denoted as $\overline{\mathbb{B}}_{\epsilon}(x)$. The unit sphere on $\mathbb{R}^n$ is denoted as $\mathbb{S}^{n-1} \triangleq \{x \in \mathbb{R}^n \ \vert\ \|x\| = 1\}$. We let $\overline{\mathbb{R}} \triangleq \mathbb{R} \cup \{\pm \infty\}$. A function $f \colon \mathbb{R}^n \rightarrow \mathbb{R}$, is said to be $L-$Lipschitz if for every $x, x' \in \mathbb{R}^n$ we have $|f(x) - f(x')| \leq L \|x-x'\|$. Given two real-valued functions $f,\ g$ on $\mathbb{R}^n$, we denote their \emph{integral convolution} as $\left(f * g\right)(x) \triangleq \int_{\mathbb{R}^n}f(\tau)g(x-\tau)d\tau \equiv \int_{\mathbb{R}^n} f(x-\tau)g(\tau)d\tau$ (assuming it is well-defined). Associated with integral convolution, and given a function $f \colon \mathbb{R}^n \rightarrow \mathbb{R}$, we define the \emph{dilation operation} as $(\lambda \bullet f)(x) = \lambda^{n} f(\lambda x),$ for any $\lambda > 0$, noting that this operation dilates $f$, compressing it towards the origin without altering its integral over $\mathbb{R}^n$.
\par Given a closed and proper function $f \colon \mathbb{R}^n \rightarrow \overline{\mathbb{R}}$, we define its proximity operator as $\textbf{prox}_{\lambda f}(x) \triangleq \arg\min_{w \in \mathbb{R}^n} \{f(w) + 1/(2\lambda)\|w-x\|^2 \},$ where $\lambda > 0$. If $\textbf{prox}_{\lambda f}(x)$ can be computed expeditiously (e.g., in closed-form), we say that $f$ is \emph{proximable}. Similarly, we define the \emph{Moreau envelope} of $f$ as $e_{\lambda}f(x) \triangleq \inf_{w \in \mathbb{R}^n} \{f(w) + 1/(2\lambda)\|w-x\|^2\}$. For some $\rho \geq 0$, we define the space of $\rho-$weakly convex functions as
\[ \Gamma_{\rho}(\mathbb{R}^n) \triangleq \left\{f \colon \mathbb{R}^n \rightarrow \overline{\mathbb{R}}\ \big\vert\ f\text{ is proper, closed, and }f+\frac{\rho}{2}\|\cdot\|^2\text{ is convex}\right\},\]
\noindent noting that $\Gamma_0(\mathbb{R}^n)$ denotes the set of closed, proper, and convex functions.
\par Given a proper and closed function $f \colon \mathbb{R}^n \rightarrow \overline{\mathbb{R}}$, we define the \emph{regular subdifferential} of $f$ at $\bar{x} \in \mathbb{R}^n$, denoted as $\hat{\partial}f(\bar{x})$, as the set of all vectors $v \in \mathbb{R}^n$ that satisfy
\[f(x) \geq f(\bar{x}) + v^\top (x-\bar{x}) + o\left( \|x-\bar{x}\|\right).\]
\noindent The \emph{limiting subdifferential} of $f$ at $\bar{x} \in \mathbb{R}^n$, denoted as $\partial f(\bar{x})$, is defined as the set of vectors $v \in \mathbb{R}^n$ for which there exist sequences $x_k \rightarrow_f \bar{x}$ and $v_k \in \hat{\partial}f(x_k)$, with $v_k \rightarrow v$, where $x \rightarrow_f \bar{x}$ denotes $f-$attentive convergence. Finally, we denote the Clarke subdifferential of $f$ at $\bar{x}$ as $\bar{\partial}f(\bar{x})$. If $f$ is subdifferentially regular at $\bar{x}$, we have that $\partial f(\bar{x}) = \hat{\partial} f(\bar{x}) = \bar{\partial}f(\bar{x})$ (e.g., this holds for any $f \in \Gamma_{\rho}(\mathbb{R}^n)$).

\section{Preliminaries} \label{sec: preliminaries}
\subsection{Clarke and Goldestein subdifferentials}

\par We begin our discussion by characterizing the Clarke subdifferential for Lipschitz functions. Its construction relies on the fact that, due to Rademacher's theorem, any Lipschitz function is almost everywhere differentiable (i.e., the subset of $\mathbb{R}^n$ in which $f$ is non-differentiable has Lebesgue measure zero). This is done in the following proposition, which is due to Clarke \cite{Clarke}.
\begin{proposition}[Clarke subdifferential characterization (Lipschitz functions)  \cite{Clarke}]
Let $f \colon \mathbb{R}^n \rightarrow \mathbb{R}$ be an $L-$Lipschitz function, for some $L > 0$. Then, for any $x \in \mathbb{R}^n$ and any $g \in \bar{\partial} f(x)$, we have that $\|g\| \leq L$, and the set-valued mapping $\bar{\partial}f(\cdot)$ is upper semicontinuous. Morover, for any $x, x' \in \mathbb{R}^n$, there exists $\lambda \in (0,1)$ and $g \in \bar{\partial}f(\lambda x+ (1-\lambda)x')$, such that $f(x) - f(x') = g^\top (x-x')$. Finally,
\[ \bar{\partial} f(x) \equiv \text{conv}\left(\left\{g \in \mathbb{R}^n\ \vert\ g = \lim_{x_k \rightarrow x} \nabla f(x_k) \right\}\right),\]
\noindent i.e., the Clarke subdifferential is the convex hull of all limit points of $\nabla f(x_k)$ over all sequences $\{x_k\}_{k=0}^{\infty}$ of differentiable points of $f(\cdot)$ which converge to $x$.
\end{proposition}
\par Consider the minimization of a general Lipschitz continuous function $f$. It is known that finding an $\epsilon-$Clarke stationary point, in the sense that we have found an $x$ such that $\min\{\|g\|\  \vert\ g \in \partial f(x)\} \leq \epsilon$, is intractable (see \cite{pmlr:Zhang_etal}). Instead, it has been observed that a relaxation of $\epsilon-$Clarke stationarity, known as the $(\mu,\epsilon)-$Goldstein stationarity (see Section \ref{sec: termination criteria} for a definition), is computationally tractable. This relies on the so-called $\mu-$Goldstein subdifferential, which we define next.
\begin{definition}[$\mu-$Goldstein subdifferential \cite{Goldstein}] \label{def: Goldstein subdifferential}
Let $f \colon \mathbb{R}^n \rightarrow \mathbb{R}$ be an $L-$Lipschitz function. Given any $x \in \mathbb{R}^n$, the $\mu-$Goldstein subdifferential of $f$ at $x$ is defined by $\bar{\partial}_{\mu} f(x) \triangleq \text{conv}\left(\cup_{y \in \overline{\mathbb{B}}_{\mu}(x)} \bar{\partial} f(y)\right)$, where $\mu > 0$ is a positive constant. 
\end{definition}
\noindent We revisit the notion of generalized approximate stationarity in the context of nonsmooth and nonconvex optimization in Section \ref{sec: termination criteria}.

\subsection{Uniform randomized smoothing}
\par We next introduce the notion of uniform randomized smoothing, which is obtained by the operation of integral convolution and dilation. To that end, we let $g:\mathbb{R}^n \rightarrow \mathbb{R}$ be an \emph{integral smoothing kernel}, i.e., a bounded piecewise continuous density function (i.e., $\int_{\mathbb{R}^n} g(x)dx = 1$) that is even (i.e., $g(-x) = g(x)$), and satisfies
\[ \int_{\mathbb{R}^n} \|x\| g(x) dx < +\infty.\] 
\noindent Then, given some $L-$Lipschitz function $f \colon \mathbb{R}^n \rightarrow \mathbb{R}$, we define the surrogate function $f_{\mu} \colon \mathbb{R}^n \rightarrow \mathbb{R}$ as
\begin{equation*}
\begin{split}
f_{\mu}(x) & = \left(f * \left(\mu^{-1}\bullet g\right)\right)(x) \equiv \int_{\mathbb{R}^n} f(x-t) \left(\mu^{-1}\bullet g\right)(t)dt \\
& = \int_{\mathbb{R}^n} f(x-\mu t) g(t)dt \equiv \mathbb{E}_g\left\{ f(x+\mu t)\right\},
\end{split}
\end{equation*}
\noindent where we used a change of variables and the last equivalence follows from the symmetry of $g$. In this paper, we will focus on a particular mollifier function $g$, namely, the probability density function (p.d.f.) of a uniform random vector $U$ over the closed unit ball on $\mathbb{R}^n$ (i.e., over $\overline{\mathbb{B}}_1(0_n)$). Specifically, let $U \sim \mathrm{U}\left(\overline{\mathbb{B}}_1(0_n) \right)$. Then, the p.d.f., say $g$, of $U$ reads:
\[ g_n(u) = \begin{cases}
\frac{1}{c_n}, & \text{if }\|u\| \leq 1,\\
0, &\text{otherwise}
\end{cases},\ \text{with } c_n \triangleq \frac{\pi^{n/2}}{\Gamma(n/2+1)},\ \Gamma(n/2+1) \triangleq \begin{cases} (n/2)!,& \text{if }n\text{ is even}\\
\sqrt{\pi}\frac{n!!}{2^{(n+1)/2}},&\text{if }n\text{ is odd}\end{cases},\]
\noindent where $n!! = n(n-2)\cdots 2$ if $n$ is even and $n!! = n (n-2)\cdots 1$, if $n$ is odd. Applying the dilation operation on $g$ with a constant $\mu^{-1} > 0$ yields the p.d.f. of a uniform random variable over the $\mu-$closed ball, i.e., 
\[(\mu^{-1} \bullet g)(u) =  \begin{cases}
\frac{1}{c_n \mu^n}, & \text{if }\|u\| \leq \mu,\\
0, &\text{otherwise}
\end{cases}.\]
Next, we provide a well-known key result which showcases the smoothing effect of integral convolution, under the assumption of Lipschitz continuity of $f$.

\begin{lemma}[Uniform randomized smoothing of Lipschitz functions] \label{lemma: integral smoothing via mollifiers}
Let $f \colon \mathbb{R}^n \rightarrow \mathbb{R}$ be an  $L-$Lipschitz continuous function, and let $g \colon \mathbb{R}^n \rightarrow \mathbb{R}$ be the p.d.f. of a uniform random variable, say $U \colon \Omega \rightarrow \mathbb{R}^n$, over the $n-$dimensional unit ball $\overline{\mathbb{B}}_1(0_n)$, i.e., $U \sim \mathrm{U}\left(\overline{\mathbb{B}}_1(0_n) \right)$. Then, the surrogate function defined as
\[ f_{\mu}(x) = \left(f * \left(\mu^{-1}\bullet g\right)\right)(x),\qquad \text{for all }x \in \mathbb{R}^n,\]
\noindent satisfies the following:
\begin{itemize}
\item $f_{\mu}$ is $L-$Lipschitz continuous and $|f_{\mu}(x) - f(x)| \leq \mu L$, for all $x \in \mathbb{R}^n$;
\item $f_{\mu}$ is $\frac{cL \sqrt{n}}{\mu}-$Lipschitz smooth, where $c > 0$ is a bounded constant independent of $n$. Moreover, we have that
\begin{equation*}
\begin{split}
\nabla f_{\mu}(x) & = \frac{n}{\mu}\mathbb{E}_{W \sim \mathrm{U}\left(\mathbb{S}^{n-1}\right)}\left\{f(x+\mu W)W\right\} \\
 & \equiv \frac{n}{2\mu}\mathbb{E}_{W \sim \mathrm{U}\left(\mathbb{S}^{n-1}\right)}\left\{\left(f(x+\mu W) - f(x-\mu W)\right)W\right\},
 \end{split}
 \end{equation*}
 \noindent where, as indicated above, $W$ is a uniform random variable over the $n-$dimensional unit sphere $\mathbb{S}^{n-1}$;
\item For all $x \in \mathbb{R}^n$, we have that $\nabla f_{\mu}(x) \in \bar{\partial}_{\mu} f(x)$, where $\bar{\partial}_{\mu} f(x)$ is the $\mu-$Goldestein subdifferential of $f$ at $x$.
\end{itemize}

 Moreover, at every $\bar{x} \in \mathbb{R}^n$, we have
\[ \bar{\partial} f(\bar{x}) = \text{conv}\left(\limsup_{x \rightarrow \bar{x}, \mu \searrow 0} \nabla f_{\mu}(x)\right),\]
\noindent i.e., gradient consistency holds, noting that the outer limit is defined as
\[\limsup_{x \rightarrow \bar{x}, \mu \searrow 0} \nabla f_{\mu}(x) \triangleq \left\{v\ \vert\ \exists\ \{(x_k,\mu_k)\}_{k \in \mathbb{N}} \rightarrow (\bar{x},0),\ \text{such that } \nabla f_{\mu_k}(x_k) \rightarrow v\right\}.\]
\end{lemma}
\begin{proof}
The first part of the lemma follows from \cite[Proposition 2.2 and Theorem 3.1]{NEURIPS2022_Linetal}, where the expression for the gradient of $f_{\mu}$ can be shown as in \cite[Lemma 2.1]{Flaxman_etal} (where the symmetric expression for the gradient is given in \cite{JMLR:Shamir}; see also \cite[Section 9.4]{nemirovskiyudin1983}). The gradient consistency property follows from a trivial extension of \cite[Theorem 9.67]{Rockafellar2009VarAn}.
\end{proof}

\subsection{Inexact noisy oracle for function evaluations} \label{subsec: inexact noisy oracle}
\par One crucial part of this work is that we do not assume the availability of an unbiased (exact) stochastic oracle for evaluating the function $F(x,\xi)$, give some $x \in \mathbb{R}^n$ and some $\xi \in \Xi$. Instead, we make use of a general inexact stochastic oracle, defined below.
\begin{definition}[Inexact noisy oracle] \label{def: inexact noisy oracle}
Let Assumption \textnormal{\ref{assum: basic assumption}} hold for problem \eqref{eqn: main problem}. For any $x \in \mathbb{R}^n$ and any $\xi \in \Xi$, we assume the availability of an \emph{inexact noisy oracle} which returns a measurable function $\tilde{F}(x,\xi) \triangleq F(x,\xi) + \delta(x,\xi)$, where $\delta(\cdot,\cdot)$ is some measurable random error function, such that $|\delta(x,\xi)| \leq \tilde{\delta}$, for all $x \in \mathbb{R}^n$ and a.e. $\xi \in \Xi$, where $\tilde{\delta} > 0$ is some positive constant.
\end{definition}

\begin{remark} \label{remark: discussion on oracle definition}
As we will see later on, when discussing the applications of the proposed methodology, Definition \textnormal{\ref{def: inexact noisy oracle}} is consistent with what we are trying to achieve in this paper. Specifically, we assume that this ``oracle" is a numerical method that is employed in order to evaluate $F(x,\xi)$ up to some error tolerance, say $\tilde{\delta} > 0$, in the sense that for all $x \in \mathbb{R}^n$ and any $\xi \in \Xi$, it returns a quantity that satisfies 
$$\left|\tilde{F}(x,
\xi) - F(x,\xi)\right| = |\delta(x,\xi)| \leq \tilde{\delta}.$$
\par Saying that this oracle returns a measurable \emph{function} is effectively the same as assuming that the numerical algorithm that we employ is deterministic or stochastic with a fixed seed, in the sense that it always returns the same result for fixed $x$ and $\xi$. If the underlying numerical algorithm (constituting the oracle) were stochastic, then we would instead have to treat $\delta(\cdot,\cdot)$ as a measurable multifunction. This is omitted for simplicity of exposition.
\par Finally, the imposition of a uniform error bound $\tilde{\delta}$ on the oracle error is also made for simplicity. Indeed, one could instead assume that $\tilde{\delta}$ is a random variable with finite first- and second-moments. This is also omitted for brevity of exposition.
\end{remark}

\begin{assumption} \label{assum: second assumption} 
\noindent Either of the following two conditions is in effect for \textnormal{\eqref{eqn: main problem}}:
\begin{itemize}
\item[\textbf{(B1)}] The inexact noisy oracle is such that for any $x \in \mathbb{R}^n$ we have
\[ \mathbb{E}_{\xi}\left\{\delta(x,\xi)\right\} = \Delta,\]
\noindent where $\Delta$ is some constant;
\item[\textbf{(B2)}] The function $r$ appearing in \eqref{eqn: main problem} is such that $r = h + \iota_{\mathcal{X}}$, where $h \in \Gamma_0(\mathbb{R}^n)$, $\mathcal{X}$ is a convex and compact set with diameter $D>0$, and $\iota(\cdot)$ is the indicator function defined as $\iota_{\mathcal{X}}(x) = 0,$ if $x \in \mathcal{X}$, and $+\infty$ otherwise.
\end{itemize}
\end{assumption}
\begin{remark} \label{remark: inexact oracle}
\par Let us briefly discuss the two conditions laid out in Assumption \textnormal{\ref{assum: second assumption}}. As stated, we only require one of these two conditions to hold. Condition \textnormal{\textbf{(B1)}} is very mild, and implies that the oracle error has a first-order moment independent of $x$. This is very natural for oracles considered in this work. To see this, assume, for example, that $F(x,\xi) = \min_{y \in \mathcal{Y}(x,\xi)} G(x,y,\xi)$, where $\mathcal{Y}(x,\xi) \subset \mathbb{R}^m$, in which case evaluating $F$ requires the solution of an optimization problem (noting that similar problems are considered in this paper and that the discussion can symmetrically apply to maximization problems as well). Then, the oracle is a numerical algorithm that performs this optimization up to some prespecified tolerance, say $\tilde{\delta}$, in the sense that it returns $\tilde{F}(x,\xi)$, such that
\[ \tilde{F}(x,\xi) - F(x,\xi) \leq \tilde{\delta},\]
\noindent for all $x \in \mathbb{R}^n$ and any $\xi \in \Xi$. Since this tolerance is independent of $x$, we will always observe (for any $x \in \mathbb{R}^n$) that $\tilde{F}(x,\xi) = F(x,\xi) + \delta(x,\xi)$, where, in this example, $\delta(x,\xi) \leq \tilde{\delta}$ is some positive random variable. The assumption relies on the intuition that the range of values of $\delta(\cdot,\cdot)$ should be independent of $x$, and requires that $\mathbb{E}_{\xi}\{\delta(x,\xi)\} = \Delta$, for any $x \in \mathbb{R}^n$. This is consistent with practice, assuming that the optimization algorithm (oracle) terminates after reaching an optimality gap of $\tilde{\delta}$, irrespectively of $x$. Indeed, in that case, and for any fixed $x$, the oracle should return an evaluation such that $\tilde{F}(x,\xi)-F(x,\xi) \leq \tilde{\delta}$ and thus averaging those evaluation differences over $\xi$ should yield some constant $\Delta$, and the intuition behind condition \textnormal{\textbf{(B1)}} is that this constant should not depend on $x$ (something that would naturally hold if, e.g., the minimization problem with respect to $y$ were convex).
\par If we assume that condition \textnormal{\textbf{(B1)}} does not hold, we instead assume, in condition \textnormal{\textbf{(B2)}}, that $x$ is constrained on a convex and compact set, with bounded diameter, say $D > 0$.
\end{remark}

\section{Termination criteria for nonsmooth and nonconvex stochastic composite optimization} \label{sec: termination criteria}
\par In the context of nonsmooth and nonconvex optimization (e.g., see problem \eqref{eqn: main problem}), it is important to establish a practical and useful metric for measuring progress of an optimization algorithm. To that end, we list two important notions which will be useful in this work, and based upon which we will derive a novel notion of approximate stationarity suitable for nonsmooth and nonconvex stochastic composite optimization of the form of \eqref{eqn: main problem}.\\
\paragraph{\textbf{$(\mu,\epsilon)-$Goldstein stationary points}}
\par Consider problem \eqref{eqn: main problem}, and assume that $r = 0$. In the context of nonsmooth and nonconvex optimization of Lipschitz continuous functions, it is well-known that finding an approximate Clarke-stationary point of $f$ (or, equivalently, an $\epsilon-$Clarke stationary point of $f$), i.e., a  point $x \in \mathbb{R}^n$ such that
\[\min\{\|g\|\ \vert\ g \in \bar{\partial}f(x)\} \leq \epsilon,\]
\noindent where $\epsilon>0$ is some pre-specified tolerance, is intractable (e.g., see \cite{pmlr:Zhang_etal}). Instead, following \cite{pmlr:Zhang_etal}, one typically utilizes the notion of $(\mu,\epsilon)-$Goldstein stationary points. Specifically, a point $x \in \mathbb{R}^n$ is said to be a $(\mu,\epsilon)-$Goldstein stationary point if the following inequality is satisfied:
\[ \min\{\|g\|\ \vert\ g \in \bar{\partial}_{\mu} f(x)\} \leq \epsilon.\]

\paragraph{\textbf{$(\mu,\epsilon)-$Moreau envelope stationary points}}
\par Next, let us consider problem \eqref{eqn: main problem}, and assume that $f$ is $\rho-$weakly convex, for some $\rho > 0$ (i.e., $f \in \Gamma_{\rho}(\mathbb{R}^n)$). In this case, the Moreau envelope of the composite objective function, $e_{\mu} \phi$, is well-defined and continuously differentiable for all $\mu < \rho^{-1}$, and is known to serve as a good measure of near-stationarity (see \cite{DavisDrus_SIOPT}). Specifically, we say that $x \in \mathbb{R}^n$ is an $(\mu,\epsilon)-$Moreau envelope stationary point for \eqref{eqn: main problem} if it satisfies:
\[ \left\|\nabla e_{\mu}\phi(x)\right\| \leq \epsilon,\]
\noindent with the assumption that $\mu < \rho^{-1}$. In that case, one can show that if $x$ is a $(\mu,\epsilon)-$Moreau envelope stationary point, then it is close to an $\epsilon-$Clarke stationary point of $\phi$. Specifically, if we let $\hat{x} = \textbf{prox}_{\mu \phi}(x)$, then the following relations hold:
\[\begin{cases}
\|x-\hat{x}\| = \mu \|\nabla e_{\mu}\phi(x)\|&\\
\phi(\hat{x}) \leq \phi(x)&\\
\min\{\|g\|\ \vert\ g \in \bar{\partial} \phi(\hat{x})\} \leq \|\nabla e_{\mu}\phi(x)\|
\end{cases}.\]
 \paragraph{\textbf{Surrogate $(\lambda,\mu,\epsilon)-$Moreau envelope stationary points}}
 \par In this work, we will, implicitly, make use of both notions of near stationarity. Indeed, we do not assume that $f$ in \eqref{eqn: main problem} is weakly convex, and thus we cannot make direct use of the $(\mu,\epsilon)-$Moreau envelope stationarity. On the other hand, we consider problems for which $r \neq 0$ and it is allowed to be extended-valued (and thus not Lipschitz continuous). As a result we cannot make direct use of the $(\mu,\epsilon)-$Goldstein stationarity. Instead, we will attempt to generalize both notions and combine them using an appropriate surrogate function. Specifically, by utilizing a smooth surrogate of $f$ based on randomized uniform smoothing, we focus on solving the following $\rho-$weakly convex optimization problem:
\begin{equation} \label{eqn: surrogate problem} \tag{$\text{P}_{\mu}$}\min_{x \in \mathbb{R}^n} \phi_{\mu}(x) \triangleq f_{\mu}(x) +  r(x),\qquad f_{\mu}(x) \triangleq \mathbb{E}_{\xi,U\sim \mathrm{U}(\overline{\mathbb{B}}_1(0))} \left\{F(x+\mu U,\xi)\right\},\end{equation}
\noindent where, using Lemma \ref{lemma: integral smoothing via mollifiers} and Assumption \ref{assum: basic assumption}, we have that $\phi_{\mu} \in \Gamma_{\rho}(\mathbb{R}^n)$, with $\rho = cG\sqrt{n}\mu^{-1}$ (see also Lemma \ref{lemma: properties of f-mu}). Using these facts, we are now able to define the proposed notion of a \emph{surrogate $(\lambda,\mu,\epsilon)-$Moreau envelope stationary point}.
\begin{definition}[Surrogate $(\lambda,\mu,\epsilon)-$Moreau envelope stationarity] \label{def: surrogate Moreau envelope stationarity}
Consider problem \eqref{eqn: main problem} and let $x \in \textnormal{dom } r$. We say that $x$ is a \emph{surrogate $(\lambda,\mu,\epsilon)-$Moreau envelope stationary point} for \eqref{eqn: main problem} if for some $\epsilon > 0$ and some $\mu > 0$, it holds that
\[ \left\| \nabla e_{\lambda} \phi_{\mu}(x)\right\| \leq \epsilon,\]
\noindent with $\lambda < \rho^{-1}$ and $\rho = cG \sqrt{n} \mu^{-1}$, where $\phi_{\mu}$ is defined as in \eqref{eqn: surrogate problem}, and $\rho$ is the weak convexity constant of $\phi_{\mu}$.
\end{definition}
\par We now proceed to explain why the surrogate $(\lambda,\mu,\epsilon)-$Moreau envelope stationarity is a suitable approximate stationarity condition for \eqref{eqn: main problem}.  We start by noting that for any $x \in \text{dom } r \equiv \text{dom } \phi_{\mu}$,
\[ \bar{\partial} \phi_{\mu}(x) = \hat{\partial}\phi_{\mu}(x) = \partial \phi_{\mu}(x) = \nabla f_{\mu}(x) + \partial r(x), \]
\noindent where the first three equalities follow from the fact that $\phi_{\mu}$ is subdifferentially regular (as a weakly convex function; see \cite[Example 10.32]{Rockafellar2009VarAn}) and the second equality follows from \cite[Exercise 8.8]{Rockafellar2009VarAn}. 
\par Let $x^* \in \text{dom } \phi_{\mu}$ satisfying $\|\nabla e_{\lambda}\phi_{\mu}(x^*)\|\leq \epsilon$ for some $\epsilon > 0$, such that $\lambda < \rho^{-1} = \mu/(cG\sqrt{n})$. Then, we observe that if $\hat{x} = \textbf{prox}_{\lambda \phi_{\mu}}(x^*)$, we have that
\[ \|x^*- \hat{x}\| \leq \lambda \epsilon,\qquad \min\{\|g\|\ \vert\ g \in \bar{\partial}\phi_{\mu}(\hat{x})\} \leq \epsilon.\]
\noindent Hence, we may observe that
\[ \min\{ \|g\|\ \vert\ g \in \bar{\partial}_{\lambda \epsilon} \phi_{\mu}(x^*)\} \leq \epsilon.\]
\noindent Indeed, since $\|\hat{x} - x^*\| \leq \lambda \epsilon$, and $\bar{\partial}_{\lambda \epsilon} \phi_{\mu}(x^*) = \text{conv}\left(\cup_{y \in \overline{\mathbb{B}}_{\lambda \epsilon}(x^*)} \bar{\partial} \phi_{\mu}(y) \right)$, we have that 
\[ \hat{x} \in \overline{\mathbb{B}}_{\lambda \epsilon}(x^*),\qquad \min\{\|g\|\ \vert\ g \in \bar{\partial} \phi_{\mu}(\hat{x})\} \leq \epsilon.\]
\noindent Thus, it must be true that $x^*$ is a $(\lambda \epsilon,\epsilon)-$Goldstein stationary point for $\phi_{\mu}$. For example, and without loss of generality, we may assume that $\lambda\epsilon \leq \mu$, in which case $x^*$ is a $(\mu,\epsilon)-$Goldstein stationary point for $\phi_{\mu}$ (noting that, in general, we expect that $\lambda\epsilon \ll \mu)$.
\par Let us now see how this translates to the original problem in the special case where $r = 0$. In that case, we have, from \cite[Lemma 4]{JMLR:KornowskiShamir}, that
\[ \bar{\partial}_{\lambda \epsilon} \phi_{\mu}(x^*) \subseteq \bar{\partial}_{\lambda \epsilon + \mu} \phi(x^*),\]
\noindent or, in other words, if $r = 0$ we obtain that $x^*$ is a $(\lambda \epsilon + \mu,\epsilon)-$Goldstein stationary point for $\phi$.
\begin{remark}
Let us also briefly discuss the relation of the Goldstein subdifferential of the surrogate function $\phi_{\mu}$ to the original function $\phi$. Since $\phi_{\mu}$ is $\rho-$weakly convex, with $\rho = cG\sqrt{n}\mu^{-1}$, then (using Lemma \ref{lemma: integral smoothing via mollifiers}) any $v \in \bar{\partial} \phi_{\mu}(x)$ satisfies, for all $x,\ y \in \text{dom } \phi_{\mu} \equiv \text{dom }\phi$:
\begin{equation*}
    \begin{split}
        \phi(y) + \mu G \geq \phi_{\mu}(y) &\geq \phi_{\mu} + v^\top (y-x) - \frac{\rho}{2}\|y-x\|^2 \\
        &\geq \phi(x) + v^\top (y-x) - \frac{\rho}{2}\|y-x\|^2,
    \end{split}
\end{equation*}
\noindent thus implying that $v \in \bar{\partial}^{\text{eps}}_{\mu G} \phi(x)$, where $\bar{\partial}^{\text{eps}}_{\mu G} \phi(x)$ denotes the \textit{epsilon-(regular) subdifferential} of $\phi$ at $x$, with $G$ being the Lipschitz continuity constant of $\phi$.
\par Hence, for any $x \in \text{dom }\phi$, we can easily show that the $\lambda\epsilon-$Goldstein subdifferential of the surrogate $\phi_{\mu}$ (which is essentially employed in our proposed approximate stationarity condition) satisfies
\[ \bar{\partial}_{\lambda \epsilon} \phi_{\mu}(x) = \text{conv}\left(\cup_{y \in \overline{\mathbb{B}}_{\lambda \epsilon}(x)} \bar{\partial} \phi_{\mu}(y)\right) \subseteq \text{conv}\left(\cup_{y \in \overline{\mathbb{B}}_{\lambda \epsilon}(x)} \bar{\partial}^{\text{eps}}_{\mu G} \phi(y)\right).\]
\noindent In other words, the $\lambda \epsilon-$Goldstein subdifferential of the surrogate $\phi_{\mu}$, at some $x \in \text{dom }\phi$, is a subset of an enlargement of the $\lambda \epsilon-$Goldestein subdifferential of $\phi$, obtained by substituting, in the definition of the Goldstein subdifferential, the regular subdifferential by the $\mu G-$epsilon-regular subdifferential (as defined above).  Although the latter construct might be an artificially large set in general, our discussion above showcases that the particular subset obtained by our proposed approximate stationarity condition is a sensible and appropriate choice in the context of nonsmooth and nonconvex stochastic composite optimization problems considered herein.
\end{remark}

\section{An inexact zeroth-order method for \eqref{eqn: main problem}} \label{sec: inexact zeroth-order method}
\par We are now ready to derive and analyze our proposed algorithm for the solution of \eqref{eqn: main problem}, under Assumption \ref{assum: basic assumption}, where $F$ is only accessible via an inexact noisy oracle (cf. Definition \ref{def: inexact noisy oracle}). The algorithm will be analyzed under two general assumptions, by further imposing Assumption \ref{assum: second assumption} (i.e., either imposing a natural first-moment stationarity property of the oracle, or assuming that the optimization of \eqref{eqn: main problem} is performed over a convex and compact set $\mathcal{X}$). We begin by stating the proposed method in Algorithm \ref{Algorithm: Z-iProxSG}.
\renewcommand{\thealgorithm}{Z-iProxSG}

\begin{algorithm}[!ht]
\caption{Zeroth-order inexact Proximal Stochastic Gradient method}
    \label{Algorithm: Z-iProxSG}

\begin{algorithmic}
\State \textbf{Input:}  $x_0 \in \textnormal{dom}(r)$, a sequence $\{\alpha_t\}_{t \geq 0} \subset \mathbb{R}_+$, $\mu > 0$, and $T > 0$.
\For {($t = 0,1,2,\ldots, T$)}
\State Sample (i.i.d.) $\xi_t \in \Xi$, $W_t \sim \mathrm{U}\left(\mathbb{S}^{n-1}\right)$.
\State Compute and store two oracle evaluations: $\tilde{F}\left(x_t + \mu W_t,\xi_t\right),\ \tilde{F}(x_t-\mu W_t,\xi_t).$
\State Set
\[x_{t+1} = \textbf{prox}_{\alpha_t r}\left(x_t - \alpha_t G\left(x_t,\xi_t,W_t;\mu\right)\right), \]
\State where $G_t \equiv G\left(x_t,\xi_t,W_t;\mu\right) \triangleq \frac{n}{2\mu}\left(\tilde{F}\left(x_t + \mu W_t,\xi_t\right) - \tilde{F}(x_t-\mu W_t,\xi_t)\right) W_t$.
\EndFor
\State Sample $t^* \in \{0,\ldots,T\}$ according to $\mathbb{P}(t^* = t) = \frac{\alpha_t}{\sum_{i = 0}^T\alpha_i}$.
\State \Return $x_{t^*}$.
\end{algorithmic}
\end{algorithm}

\subsection{Technical results}
\par We begin by stating certain important technical results that will be instrumental in analyzing the non-asymptotic convergence of Algorithm \ref{Algorithm: Z-iProxSG}. We start by bounding the quantity $\mathbb{E}\{\|G_t\|^2\ \vert\ \xi_0,W_0,\ldots \xi_{t-1}, W_{t-1}\}$, where $G_t \equiv G(x_t,\xi_t,W_t;\mu)$ corresponds to the (biased and noisy) stochastic gradient estimator appearing in Algorithm \ref{Algorithm: Z-iProxSG}, at iteration $t \geq 0$. For simplicity of notation, we define the expectation operator with respect to the filtration up to time $t$ as $\mathbb{E}_{[t]}\left\{\cdot \right\} \triangleq \mathbb{E}\left\{ \cdot\ \vert\ \xi_0,W_0,\ldots \xi_{t-1}, W_{t-1}\right\}$.
\begin{lemma} \label{lemma: variance of stochastic gradient estimator}
Consider problem \textnormal{\eqref{eqn: main problem}} and let Assumption \textnormal{\ref{assum: basic assumption}} hold. Let also $\{x_t\}_{t=0}^T$ be the sequence of iterates generated by Algorithm \textnormal{\ref{Algorithm: Z-iProxSG}}. Then, we have that
\[\mathbb{E}_{[t]}\left\{\|G_t\|^2 \right\} \equiv \mathbb{E}_{[t]}\left\{\|G(x_t,\xi_t,W_t;\mu)\|^2\right\} \leq 32\sqrt{2\pi} nG^2 + 2\left(\frac{n\tilde{\delta}}{\mu}\right)^2,\]
\noindent where $G > 0$ is the constant appearing in Assumption \textnormal{\ref{assum: basic assumption}} and $\tilde{\delta}> 0$ is the oracle error bound given in Definition \textnormal{\ref{def: inexact noisy oracle}}.
\end{lemma}
\begin{proof}
The proof follows by extending the proof of \cite[Lemma E.1]{NEURIPS2022_Linetal}, upon noting that our stochastic oracle $\tilde{F}$ is noisy and biased. We start by noting that
\begin{equation*}
\begin{split}
&\mathbb{E}_{[t]}\left\{\|G_t\|^2\right\} =  \mathbb{E}_{[t]}\left\{\left\|\frac{n}{2\mu} \left(\tilde{F}(x_t + \mu W_t,\xi_t)- \tilde{F}(x_t-\mu W_t,\xi_t) \right)W_t \right\|^2\right\}\\
&\quad = \frac{n^2}{4\mu^2} \mathbb{E}_{[t]}\left\{\|W_t\|^2\left(F(x_t+\mu W_t,\xi_t) + \delta(x_t +\mu W_t,\xi_t) - F(x_t-\mu W_t,\xi_t) - \delta(x_t - \mu W_t,\xi_t)\right)^2\right\}\\
&\quad \leq \frac{n^2}{2\mu^2} \bigg( \mathbb{E}_{[t]}\left\{\|W_t\|^2\left(F(x_t+\mu W_t,\xi_t) - F(x_t-\mu W_t,\xi_t)\right)^2 \right\} \\
&\quad \qquad + \mathbb{E}_{[t]}\left\{ \|W_t\|^2\left(\delta(x_t +\mu W_t,\xi_t) - \delta(x_t - \mu W_t,\xi_t)\right)^2\right\} \bigg),
\end{split}
\end{equation*}
\noindent where the inequality follows from the identity $(a+b)^2 \leq 2a^2 + 2b^2$. In order to bound the first term in the right hand side of the previous inequality, we follow exactly the analysis in \cite[Lemma E.1]{NEURIPS2022_Linetal}, to obtain that
\begin{equation} \label{eqn: Variance lemma, bound 1}
\mathbb{E}_{[t]}\left\{\|W_t\|^2\left(F(x_t+\mu W_t,\xi_t) - F(x_t-\mu W_t,\xi_t)\right)^2 \right\} \leq \frac{64\sqrt{2\pi} \mu^2 G^2}{n}.
\end{equation}
\noindent For the second term, we use the definition of the oracle (cf. Definition \ref{def: inexact noisy oracle}) to obtain
\begin{equation} \label{eqn: Variance lemma, bound 2}
\mathbb{E}_{[t]}\left\{\|W_t\|^2\left(\delta(x_t+\mu W_t,\xi_t) - \delta(x_t-\mu W_t,\xi_t)\right)^2 \right\} \leq 4\tilde{\delta}^2,
\end{equation}
\noindent where we used the identity $(a-b)^2 \leq 2a^2 + 2b^2$, the fact that $\|W_t\|^2 = 1$ (since $W_t \sim \mathrm{U}(\mathbb{S}^{n-1})$), as well as the uniform bound on the oracle noise, $|\delta(x,\xi)| \leq \tilde{\delta}$, for any $x \in \mathbb{R}^n$ and a.e. $\xi \in \Xi$. By combining \eqref{eqn: Variance lemma, bound 1} and \eqref{eqn: Variance lemma, bound 2}, we obtain the result.
\end{proof}

\begin{lemma} \label{lemma: inner product with iterates and estimator bound}
Consider problem \textnormal{\eqref{eqn: main problem}} and let Assumption \textnormal{\ref{assum: basic assumption}} hold. Let also $\{x_t\}_{t=0}^T$ be the sequence of iterates generated by Algorithm \textnormal{\ref{Algorithm: Z-iProxSG}}. Then, for any $r \in \mathbb{R}^n$, we have that:
\begin{itemize}
\item If condition \textnormal{\textbf{(B1)}} of Assumption \textnormal{\ref{assum: second assumption}} holds,
\[ \mathbb{E}_{[t]}\left\{r^\top G_t \right\} = \mathbb{E}_{[t]}\left\{r^\top \nabla f_{\mu}(x_t) \right\};\]
\item If, instead, condition \textnormal{\textbf{(B2)}} of Assumption \textnormal{\ref{assum: second assumption}} holds, and $r = x_1 -x_2$, for $x_1,\ x_2 \in \mathcal{X}$, 
\[ \mathbb{E}_{[t]}\left\{r^\top G_t \right\} \geq \mathbb{E}_{[t]}\left\{r^\top \nabla f_{\mu}(x_t) \right\} - \frac{n}{\mu} \tilde{\delta} D,\]
\noindent where $D$ is the diameter of $\mathcal{X}$.
\end{itemize}
\end{lemma}
\begin{proof}
We begin by proving the first part of the result, which relies on condition \textbf{(B1)} of Assumption \ref{assum: second assumption}, which implies that for any $x \in \mathbb{R}^n$, $\mathbb{E}_{\xi}\{\delta(x,\xi)\} = \Delta$, for some constant $\Delta$. Then, for any $r \in \mathbb{R}^n$, we have
\begin{equation*}
\begin{split}
\mathbb{E}_{[t]}\left\{r^\top G_t \right\} =  &\ \frac{n}{2\mu}\mathbb{E}_{[t]}\left\{r^\top \left(F(x+\mu W_t,\xi_t) - F(x-\mu W_t,\xi_t) \right)W_t \right\}\\
&\quad + \frac{n}{2\mu}\mathbb{E}_{[t]}\left\{r^\top \left(\delta(x+\mu W_t,\xi_t) - \delta(x-\mu W_t,\xi_t) \right)W_t \right\}.
\end{split}
\end{equation*}
\noindent From \cite[Lemma 8]{JMLR:Shamir} and the symmetry of $W_t$, we obtain that the first term above satisfies
\[\frac{n}{2\mu}\mathbb{E}_{[t]}\left\{r^\top \left(F(x+\mu W_t,\xi_t) - F(x-\mu W_t,\xi_t) \right)W_t \right\} = \mathbb{E}_{[t]}\left\{ r^\top \nabla f_{\mu}(x_t)\right\},\]
\noindent noting that this holds irrespectively of whether condition \textbf{(B1)} or \textbf{(B2)} of Assumption \ref{assum: second assumption} holds. For the second term, using condition \textbf{(B1)}, we have
\begin{equation*}
\begin{split}
&\frac{n}{2\mu}\mathbb{E}_{[t]}\left\{r^\top \left(\delta(x+\mu W_t,\xi_t) - \delta(x-\mu W_t,\xi_t) \right)W_t \right\} \\ &\quad = \frac{n}{2\mu}\mathbb{E}_{[t]}\left\{ \mathbb{E}_{[t]}\left\{ r^\top \left(\delta(x+\mu W_t,\xi_t) - \delta(x-\mu W_t,\xi_t) \right)W_t\ \vert\ W_t\right\}\right\} = 0.
\end{split}
\end{equation*}
\par To prove the second part of the lemma, using condition \textbf{(B2)} instead, we observe that
\begin{equation*}
\begin{split}
&\frac{n}{2\mu}\mathbb{E}_{[t]}\left\{r^\top \left(\delta(x+\mu W_t,\xi_t) - \delta(x-\mu W_t,\xi_t) \right)W_t \right\} \leq \frac{n}{2\mu}\|r\|2\tilde{\delta} \leq \frac{n}{\mu} \tilde{\delta} D,
\end{split}
\end{equation*}
\noindent where we used the fact that $|\delta(x,\xi)| \leq \tilde{\delta}$ for all $x \in \mathbb{R}^n$ and a.e. $\xi \in \Xi$, and $r = x_1-x_2$, with $x_1,\ x_2 \in \mathcal{X}$, noting that $\mathcal{X}$ has diameter $D>0$ (from condition \textbf{(B2)}). 
\end{proof}

\begin{lemma} \label{lemma: properties of f-mu}
Fix some $\mu > 0$ and let Assumption \textnormal{\ref{assum: basic assumption}} hold. Then, $\phi_{\mu} \triangleq f_{\mu} + r \in \Gamma_{\rho}(\mathbb{R}^n)$, where $\rho = \frac{c G \sqrt{n}}{\mu}$, with $c$ a bounded constant independent of $n$, and $G$ the constant given in Assumption \textnormal{\ref{assum: basic assumption}}. Moreover, if we let $\hat{x}_t \triangleq \textbf{prox}_{\bar{\rho}^{-1}\phi_{\mu}}(x_t)$, where $x_t$ is the iterate generated by Algorithm \textnormal{\ref{Algorithm: Z-iProxSG}} at time $t \geq 0$ and $\bar{\rho} \in (\rho,2\rho]$, then,
\[ \hat{x}_t = \textbf{prox}_{\alpha_t r}\left(\alpha_t \bar{\rho}x_t - \alpha_t \nabla f_{\mu}(\hat{x}_t) + \zeta_t \hat{x}_t\right),\]
\noindent where $\alpha_t$ is the step-size of Algorithm \textnormal{\ref{Algorithm: Z-iProxSG}} at time $t$ and $\zeta_t \triangleq  1-\alpha_t \bar{\rho}$.
\end{lemma}
\begin{proof}
\par We start by noting that, from Lemma \ref{lemma: integral smoothing via mollifiers}, $f_{\mu}$ is $\rho-$Lipschitz smooth (where we used the fact that $f$ is $G-$Lipschitz continuous from Assumption \ref{assum: basic assumption}), and thus $\rho-$weakly convex (see, e.g., \cite[Proposition 4.12]{Vial_WeakConvexity}). Then, from \cite[Proposition 4.1]{Vial_WeakConvexity}, we have that $\phi_{\mu} = f_{\mu} + r$ must also be $\rho-$weakly convex, since $r \in \Gamma_0(\mathbb{R}^n)$ from Assumption \ref{assum: basic assumption}. Finally, by letting $\hat{x}_t = \textbf{prox}_{\bar{\rho}^{-1}\phi_{\mu}}(x_t)$, we have, by definition, that
\begin{equation*}
\begin{split}
\alpha_t \bar{\rho}(x_t - \hat{x}_t) \in \alpha_t \partial r(\hat{x}_t) + \alpha_t \nabla f_{\mu}(\hat{x}_t) &\Leftrightarrow \alpha_t\bar{\rho}x_t - \alpha_t \nabla f_{\mu}(\hat{x}_t) + \zeta_t \hat{x}_t \in \hat{x}_t + \alpha_t \partial r(\hat{x}_t)\\
& \Leftrightarrow \hat{x}_t =\textbf{prox}_{\alpha_t r}\left(\alpha_t \bar{\rho}x_t - \alpha_t \nabla f_{\mu}(\hat{x}_t) + \zeta_t \hat{x}_t\right),
\end{split}
\end{equation*}
\noindent where $\zeta_t = 1-\alpha_t \bar{\rho}$.
\end{proof}

\subsection{Convergence analysis} \label{subsec: convergence analysis}
\par We are now ready to derive a non-asymptotic convergence analysis of Algorithm \ref{Algorithm: Z-iProxSG}. We will analyze the algorithm based on the surrogate problem \eqref{eqn: surrogate problem}, i.e. $\min_{x \in \mathbb{R}^n} \phi_{\mu}(x)$, for some fixed $\mu > 0$. Upon noting, from Lemma \ref{lemma: properties of f-mu}, that $\phi_{\mu}$ is weakly convex, the analysis will follow by extending the analysis given in \cite{Pougk_SISC}, by allowing inexact oracle evaluations. Then, we will briefly discuss conditions on the oracle error that allow us to retrieve convergence rates appearing in the literature in the context of exact stochastic oracles (matching the rates currently available only in the unconstrained case, i.e., in the case where $r = 0$).
\begin{lemma} \label{lemma: descent of the algorithm}
Consider problem \textnormal{\eqref{eqn: main problem}} and let Assumption \textnormal{\ref{assum: basic assumption}} hold. Let also $\{x_t\}_{t=0}^T$ be the sequence of iterates of Algorithm \textnormal{\ref{Algorithm: Z-iProxSG}}. Set $\bar{\rho} \in (\rho,2\rho]$, where $\rho = \frac{cG\sqrt{n}}{\mu}$ and choose $\alpha_t \in (0,\bar{\rho}^{-1}]$, for any $t \geq 0$. Then:
\begin{itemize}
\item If condition \textnormal{\textbf{(B1)}} of Assumption \textnormal{\ref{assum: second assumption}} holds, the following inequality is satisfied:
\begin{equation*}
\begin{split}
&\mathbb{E}_{[t]}\left\{\|x_{t+1}-\hat{x}_t\|^2\right\} \leq \left(1-\left( 2\alpha_t(\bar{\rho}-\rho)\right) \right)\|x_t- \hat{x}_t\|^2 + 8\alpha_t^2 \left(16 \sqrt{2\pi}nG^2 + \frac{n^2 \tilde{\delta}^2}{\mu^2}\right).
\end{split}
\end{equation*}
\item If, instead, condition \textnormal{\textbf{(B2)}} of Assumption \textnormal{\ref{assum: second assumption}} holds, then the following inequality is satisfied:
\begin{equation*}
\begin{split}
\mathbb{E}_{[t]}\left\{\|x_{t+1}-\hat{x}_t\|^2\right\} &\leq \left(1-\left( 2\alpha_t(\bar{\rho}-\rho)\right) \right)\|x_t- \hat{x}_t\|^2 \\
&\quad + 2\zeta_t \alpha_t\frac{n}{\mu}\tilde{\delta}D + 8\alpha_t^2 \left(16 \sqrt{2\pi}nG^2 + \frac{n^2 \tilde{\delta}^2}{\mu^2}\right).
\end{split}
\end{equation*}
\end{itemize}
\end{lemma}
\begin{proof}
\par By definition, we have that $\hat{x}_t = \textbf{prox}_{\bar{\rho}^{-1}\phi_{\mu}}(x_t)$. Thus,
\begin{equation*}
\begin{split}
&\mathbb{E}_{[t]}\left\{\|x_{t+1} - \hat{x}_t\|^2\right\}  = \mathbb{E}_{[t]}\left\{\left\|\textbf{prox}_{\alpha_t r}(x_t - \alpha_t G_t) - \textbf{prox}_{\alpha_t r}\left(\alpha_t \bar{\rho}-\alpha_t \nabla f_{\mu}(x_t) + \zeta_t \hat{x}_t\right)\right\|^2 \right\}\\
& \quad \leq \mathbb{E}_{[t]}\left\{\left\|(x_t - \alpha_t G_t) -\left(\alpha_t \bar{\rho}-\alpha_t \nabla f_{\mu}(x_t) + \zeta_t \hat{x}_t\right)\right\|^2 \right\}\\
&\quad  = \zeta_t^2 \|x_t - \hat{x}_t\|^2 - 2\zeta_t \alpha_t \mathbb{E}_{[t]}\left\{(x_t - \hat{x}_t)^\top \left(G_t - \nabla f_{\mu}(\hat{x}_t) \right)\right\}+ \alpha_t^2 \mathbb{E}_{[t]}\left\{\|G_t - \nabla f_{\mu}(\hat{x}_t)\|^2 \right\},
\end{split}
\end{equation*}
\noindent where the first equality follows from Algorithm \ref{Algorithm: Z-iProxSG} and from Lemma \ref{lemma: properties of f-mu}, and the inequality follows from the non-expansiveness of the proximity operator of $r$. Next, we separate two cases. 
\noindent \paragraph{\textbf{Case 1:}} Under condition \textbf{(B1)} of Assumption \ref{assum: second assumption}, by utilizing Lemmata \ref{lemma: variance of stochastic gradient estimator} and \ref{lemma: inner product with iterates and estimator bound}, we have
\begin{equation*}
\begin{split}
&\mathbb{E}_{[t]}\left\{\|x_{t+1} - \hat{x}_t\|^2\right\} \\
&\quad \leq \zeta_t^2 \|x_t - \hat{x}_t\|^2 - 2\zeta_t \alpha_t (x_t - \hat{x}_t)^\top \left(\nabla f_{\mu}(x_t) - \nabla f_{\mu}(\hat{x}_t) \right)+ 8\alpha_t^2 \left(16 \sqrt{2\pi}nG^2 + \frac{n^2 \tilde{\delta}^2}{\mu^2}\right)\\
&\quad \leq \zeta_t^2 \|x_t - \hat{x}_t\|^2 +2\zeta_t \alpha_t\rho \|x_t - \hat{x}_t\|^2+ 8\alpha_t^2 \left(16 \sqrt{2\pi}nG^2 +\frac{n^2 \tilde{\delta}^2}{\mu^2}\right) \\
&\quad = \left(1-\left( 2\alpha_t(\bar{\rho}-\rho)+\alpha_t^2\bar{\rho}(2\rho-\bar{\rho})\right) \right)\|x_t- \hat{x}_t\|^2 + 8\alpha_t^2 \left(16 \sqrt{2\pi}nG^2 + \frac{n^2 \tilde{\delta}^2}{\mu^2}\right),
\end{split}
\end{equation*}
\noindent where, in the last inequality, we used the fact that $f_{\mu}$ is $\rho-$weakly convex, with $\rho = cG\sqrt{n}\mu^{-1}$ (cf. Lemma \ref{lemma: properties of f-mu}), which implies (e.g., from \cite[Proposition 4.10]{Vial_WeakConvexity}) that
\[ (x_1 - x_2)^\top (\nabla f_{\mu}(x_1) - \nabla f_{\mu}(x_2)) \geq -\rho\|x_1-x_2\|^2,\qquad \text{for all }x_1, x_2 \in \mathbb{R}^n.\]
\noindent The first inequality then follows by noting that $\bar{\rho} \leq 2\rho$.
\noindent \paragraph{\textbf{Case 2:}} Similarly, under condition \textbf{(B2)} of Assumption \ref{assum: second assumption}, by utilizing once again Lemmata \ref{lemma: variance of stochastic gradient estimator} and \ref{lemma: inner product with iterates and estimator bound}, we have
\begin{equation*}
\begin{split}
&\mathbb{E}_{[t]}\left\{\|x_{t+1} - \hat{x}_t\|^2\right\} \\
&\  \leq \zeta_t^2 \|x_t - \hat{x}_t\|^2 - 2\zeta_t \alpha_t (x_t - \hat{x}_t)^\top \left(\nabla f_{\mu}(x_t) - \nabla f_{\mu}(\hat{x}_t) \right) + 2\zeta_t \alpha_t\frac{n}{\mu}\tilde{\delta}D \\ &\ \qquad + 8\alpha_t^2 \left(16 \sqrt{2\pi}nG^2 + \frac{n^2 \tilde{\delta}^2}{\mu^2}\right)\\
&\  \leq \zeta_t^2 \|x_t - \hat{x}_t\|^2 +2\zeta_t \alpha_t\rho \|x_t - \hat{x}_t\|^2 + 2\zeta_t \alpha_t\frac{n}{\mu}\tilde{\delta}D +  8\alpha_t^2 \left(16 \sqrt{2\pi}nG^2 + \frac{n^2 \tilde{\delta}^2}{\mu^2}\right) \\
&\  = \left(1-\left( 2\alpha_t(\bar{\rho}-\rho)+\alpha_t^2\bar{\rho}(2\rho-\bar{\rho})\right) \right)\|x_t- \hat{x}_t\|^2 + 2\zeta_t \alpha_t\frac{n}{\mu}\tilde{\delta}D + 8\alpha_t^2 \left(16 \sqrt{2\pi}nG^2 + \frac{n^2 \tilde{\delta}^2}{\mu^2}\right),
\end{split}
\end{equation*}
\noindent where $D$ is the diameter of $\mathcal{X}$ (cf. condition \textbf{(B2)} of Assumption \ref{assum: second assumption}) and $\tilde{\delta}$ is the bound on the oracle noise (cf. Definition \ref{def: inexact noisy oracle}).
\end{proof}
\par We are now ready to derive the non-asymptotic (ergodic) convergence rate of Algorithm \ref{Algorithm: Z-iProxSG} in terms of the magnitude of the gradient of the Moreau envelope of $\phi_{\mu}$. We will provide two different rates, based on either condition \textbf{(B1)} or \textbf{(B2)} of Assumption \ref{assum: second assumption}.
\begin{theorem} \label{thm:B1 convergence rate}
Fix $\mu > 0$ and consider problem \eqref{eqn: main problem} by letting Assumption \textnormal{\ref{assum: basic assumption}} hold. Let also $\{x_t\}_{t=0}^T$ be the sequence of iterates of Algorithm \textnormal{\ref{Algorithm: Z-iProxSG}}, with $x_{t^*}$ being the point returned by the method. If condition \textnormal{\textbf{(B1)}} of Assumption \textnormal{\ref{assum: second assumption}} holds, then
by letting $\Phi \geq e_{(2\rho)^{-1}}\phi_{\mu}(x_0) - \min_{x \in \mathbb{R}^n} \phi_{\mu}(x)$ (with $\Phi > 0$) and choosing 
\[ \alpha_t =\sqrt{\frac{\Phi}{8cGn^{3/2}\mu^{-1}\left(16 \sqrt{2\pi}G^2 + \frac{n}{\mu^2}\tilde{\delta}^2\right)(T+1)}},\]
\noindent we obtain that
\[ \mathbb{E}\left\{\left\| \nabla e_{(2\rho)^{-1}}\phi_{\mu}(x_{t^*})\right\|^2 \right\}\leq 4 \sqrt{\frac{8c G\Phi n^{3/2}\mu^{-1}\left(16 \sqrt{2\pi} G^2 + \frac{n}{\mu^2}\tilde{\delta}^2\right)}{T+1}},\]
\noindent where $c$ is a constant independent of $n$ (cf. Lemma \textnormal{\ref{lemma: integral smoothing via mollifiers}}), $G$ is the Lipschitz continuity constant of $f$ (cf. Assumption \textnormal{\ref{assum: basic assumption}}), and $\tilde{\delta}$ is the bound on the oracle noise (cf. Definition \textnormal{\ref{def: inexact noisy oracle}}).\\

\par Alternatively, if condition \textnormal{\textbf{(B2)}} of Assumption \textnormal{\ref{assum: second assumption}} holds instead, then, for the same choice of step-size, we obtain
\[ \mathbb{E}\left\{\left\| \nabla e_{(2\rho)^{-1}}\phi_{\mu}(x_{t^*})\right\|^2 \right\}\leq 4 \sqrt{\frac{8c G\Phi n^{3/2}\mu^{-1}\left(16 \sqrt{2\pi} G^2 + \frac{n}{\mu^2}\tilde{\delta}^2\right)}{T+1}} + 4\tilde{\delta} D\frac{n}{\mu},\]
\noindent where $D$ is the diameter of $\mathcal{X}$ (cf. Assumption \textnormal{\ref{assum: second assumption}}).
\end{theorem}
\begin{proof}
\par From Lemma \ref{lemma: descent of the algorithm}, we have that
\begin{equation*}
\begin{split}
&\mathbb{E}_{[t]}\left\{e_{\bar{\rho}^{-1}}\phi_{\mu}(x_{t+1})\right\} \leq \mathbb{E}_{[t]}\left\{\phi_{\mu}(\hat{x}_{t}) + \frac{\bar{\rho}}{2}\|\bar{x}_t-x_{t+1}\|^2\right\}\\
&\quad \leq \phi_{\mu}(\hat{x}_{t}) +\frac{\bar{\rho}}{2}\left( \|x_t- \hat{x}_t\|^2 - 2\alpha_t(\bar{\rho}-\rho)\|x_t- \hat{x}_t\|^2 + 8\alpha_t^2 \left(16 \sqrt{2\pi}nG^2 + \frac{n^2 \tilde{\delta}^2}{\mu^2}\right)\right)\\
&\quad = e_{\bar{\rho}^{-1}}\phi_{\mu}(x_t) + \bar{\rho}\left( - \alpha_t(\bar{\rho}-\rho)\|x_t- \hat{x}_t\|^2 + 4\alpha_t^2 \left(16 \sqrt{2\pi}nG^2 + \frac{n^2 \tilde{\delta}^2}{\mu^2}\right)\right),
\end{split}
\end{equation*}
\noindent where the first inequality follows from the definition of the Moreau envelope of $\phi_{\mu}$, and the equality follows from the definition of $\hat{x}_t$ (cf. Lemma \ref{lemma: properties of f-mu}). By the definition of $\hat{x}_t$, and since $\phi_{\mu}$ is $\rho-$weakly convex and $\bar{\rho} > \rho$, we may use \cite[Theorem 3.4]{ATTOUCH1993} to obtain that
\[\nabla e_{\bar{\rho}^{-1}}\phi_{\mu}(x_t) = \frac{1}{\bar{\rho}}(x_t - \hat{x}_t).\]
\noindent Using the previous fact, we next take expectation with respect to the filtration $\xi_0, W_0,\ldots,\xi_{t-1},W_{t-1}$ and use the law of total expectation to obtain
\begin{equation*}
\begin{split}
\mathbb{E}\left\{e_{\bar{\rho}^{-1}}\phi_{\mu}(x_{t+1})\right\} &\leq \mathbb{E}\left\{e_{\bar{\rho}^{-1}}\phi_{\mu}(x_t) \right\} - \frac{\alpha_t(\bar{\rho}-\rho)}{\bar{\rho}}\left\|\nabla e_{\bar{\rho}^{-1}}\phi_{\mu}(x_t)\right\|^2\\
&\qquad + 4\alpha_t^2 \bar{\rho}\left(16 \sqrt{2\pi}nG^2 + \frac{n^2 \tilde{\delta}^2}{\mu^2}\right).
\end{split}
\end{equation*}
\noindent By unrolling the above recursion, we have:
\begin{equation*}
\begin{split}
\mathbb{E}\left\{e_{\bar{\rho}^{-1}}\phi_{\mu}(x_{t+1})\right\}& \leq e_{\bar{\rho}^{-1}}\phi_{\mu}(x_0) - \frac{\bar{\rho}-\rho}{\bar{\rho}}\sum_{t=0}^T\alpha_t \left\|\nabla e_{\bar{\rho}^{-1}}\phi_{\mu}(x_t)\right\|\\ &\qquad + 4\bar{\rho}\left(16 \sqrt{2\pi}nG^2 + \frac{n^2 \tilde{\delta}^2}{\mu^2}\right) \sum_{t=0}^T \alpha_t^2.
\end{split}
\end{equation*}
\noindent We can lower bound the left-hand side of the above inequality by $\phi_{\mu}^* \triangleq \min_{x \in \mathbb{R}^n} \phi_{\mu}(x)$, and re-arrange to get
\[\frac{1}{\sum_{t=0}^T\alpha_t}\sum_{t=0}^T\alpha_t \left\|\nabla e_{\bar{\rho}^{-1}}\phi_{\mu}(x_t)\right\| \leq \frac{\bar{\rho}}{\bar{\rho}-\rho}\frac{e_{\bar{\rho}^{-1}}\phi_{\mu}(x_0)-\phi_{\mu}^* + 4\bar{\rho}\left(16 \sqrt{2\pi}nG^2 + \frac{n^2 \tilde{\delta}^2}{\mu^2}\right) \sum_{t=0}^T \alpha_t^2}{\sum_{t=0}^T\alpha_t}.\]
\noindent We observe that the left-hand side of the last inequality is nothing else than $\mathbb{E}\left\{\|\nabla e_{\bar{\rho}^{-1}}\phi_{\mu}(x_{t^*})\|^2 \right\}$, where $x_{t^*}$ is the iterate that Algorithm \ref{Algorithm: Z-iProxSG} returns.
\par To complete the proof, we set $\bar{\rho} = 2\rho$, we let $\Phi \geq e_{(2\rho)^{-1}}\phi_{\mu}(x_0)-\phi_{\mu}^*$ such that $\Phi > 0$, and set $\alpha_t = \gamma/\sqrt{T+1}$, for some $\gamma > 0$ to obtain
\[ \mathbb{E}\left\{\left\|\nabla e_{\bar{\rho}^{-1}}\phi_{\mu}(x_{t^*})\right\|^2 \right\} \leq 2\frac{\Phi+8\rho\left(16 \sqrt{2\pi}nG^2 + \frac{n^2 \tilde{\delta}^2}{\mu^2}\right)}{\gamma \sqrt{T+1}}.\]
\noindent Then, we minimize over $\gamma$, which yields that
\[ \gamma = \sqrt{\frac{\Phi}{8\rho\left(16 \sqrt{2\pi}nG^2 + \frac{n^2 \tilde{\delta}^2}{\mu^2}\right)}},\]
\noindent and completes the first part of the proof, upon noting that $\rho = cG \sqrt{n}\mu^{-1}$.
\par For the second part of the proof, we assume (without loss of generality) that $\rho^{-1} > \alpha_t$ (where $\alpha_t$ is given in the statement of the theorem). Then, $0< \zeta_t \leq 1$ (where $\zeta_t$ is defined in Lemma \ref{lemma: properties of f-mu}) and the ``descent" recursion carries an additional term of the form $4\frac{n}{\mu}\tilde{\delta}D \alpha_t$ (cf. Lemma \ref{lemma: descent of the algorithm}). The result then follows immediately by performing the same analysis as before.
\end{proof}
\begin{remark}
Let us now briefly discuss the result of Theorem \textnormal{\ref{thm:B1 convergence rate}}. To that end, we need to separate two cases, i.e., depending on whether condition \textnormal{\textbf{(B1)}} or \textnormal{\textbf{(B2)}} of Assumption \textnormal{\ref{assum: second assumption}} holds. In the former case, which is really general and highly relevant to the applications considered herein, it suffices to enforce that $\tilde{\delta} = \mathcal{O}(\mu/\sqrt{n})$ to retrieve the same convergence rate as that obtained in \textnormal{\cite[Theorem 3.2]{NEURIPS2022_Linetal}}, in the context of \emph{unconstrained} nonsmooth and nonconvex stochastic optimization. On the other hand, under condition \textnormal{\textbf{(B2)}} of Assumption \textnormal{\ref{assum: second assumption}}, we instead need to enforce that $\tilde{\delta} = \mathcal{O}(\epsilon^2 \mu/n)$ in order to retrieve the convergence rate of \textnormal{\cite[Theorem 3.2]{NEURIPS2022_Linetal}}.
\par Another important point, already briefly mentioned in Section \textnormal{\ref{subsec: inexact noisy oracle}} (cf. Remark \ref{remark: discussion on oracle definition}), is that the analysis could be extended to accommodate for stochastic upper bounds on the oracle error (i.e., allowing $\tilde{\delta}$ to be a function of the underlying randomness, and force it to have finite first- and second-moment rather than enforcing it to be uniformly bounded). This could be done using a similar methodology as that presented in \textnormal{\cite{Pougk_etal_IRSinexact}}, and would reveal an averaged-over-the-iterates error propagation. This is omitted here for brevity of exposition.
\end{remark}

\section{Selected Applications} \label{sec: applications}
\par In this section, we showcase the applicability of the proposed algorithm in two wide classes of problems, namely, two-stage stochastic programming, and stochastic minimax optimization, while also briefly mentioning certain additional applications that may be of interest to the wider academic community.
\subsection{Two-stage stochastic programming} \label{subsec: 2SP}
\par On our usual probability space $(\Omega,\mathscr{F},P)$, we consider a random vector $\xi \colon \Omega \rightarrow \Xi \subset \mathbb{R}^d$, and its induced Borel space $(\Xi,\mathscr{B}(\Xi),P)$. In this section, we consider nonconvex two-stage stochastic programming problems of the form
\begin{equation} \label{eqn: 2SP} \tag{2SP} \min_{x \in \mathcal{X}} \mathbb{E}_{\xi}\left\{\min_{y \in \mathcal{Y}(\xi)} \widehat{F}(x,y,\xi) \right\},
\end{equation}
\noindent where $\widehat{F} \colon \mathbb{R}^n \times \mathbb{R}^m \times \Xi \rightarrow \mathbb{R}$ is Borel in $\xi \in \Xi$, and continuous on $\mathbb{R}^n\times \mathbb{R}^m$ for a.e. $\xi \in \Xi$. We let $\mathcal{X} \subset \mathbb{R}^n$ be a convex and compact set, and the mulifunction $\mathcal{Y} \colon \Xi \rightrightarrows \mathbb{R}^m$ be compact-valued and measurable with respect to $\mathscr{B}(\Xi)$ (and thus, the indicator function $\iota_{\mathcal{Y}(\xi)}(\cdot)$ is random lower semicontinuous; cf. \cite[Definition 9.47]{ShapiroLecturesStochProg3rd}).

\begin{remark}
Let us note that problem \text{\eqref{eqn: 2SP}} is not stated in its full generality. Specifically, one could consider a constraint set $\mathcal{Y}(\xi)$, for the second-stage problem, which also depends $x$. Then, in order to show that Assumption \textnormal{\ref{assum: basic assumption}} holds for \text{\eqref{eqn: 2SP}}, we would have to use the perturbation analysis machinery from, e.g., \textnormal{\cite{ShapiroPerturbationAnalysis}}. This is omitted here in the interest of clarity and brevity. Nonetheless, even in its current simplified form, problem \textnormal{\eqref{eqn: 2SP}} is already very general (with a plethora of important applications, as stated in the introduction) and its solution under minimal conditions remains a challenge.
\end{remark}

\paragraph{\textbf{Regularity conditions and assumptions}}
\par In keeping with the notation of \eqref{eqn: main problem}, we let $F(x,
\xi) \triangleq \min_{y \in \mathcal{Y}(\xi)} \widehat{F}(x,y,\xi)$. In order to ensure that problem \eqref{eqn: 2SP} is well-defined, we will implicitly make the minimal assumption that $\widehat{F}(x,y^*(x,\xi(\cdot)),\xi(\cdot)) \in \mathcal{L}_1(\Omega,\mathscr{F},P;\mathbb{R})$ for any measurable selection $y^*(x,\xi(\cdot)) \in \arg\min_{y \in \mathcal{Y}(\xi(\cdot))} \widehat{F}(x,y,\xi(\cdot))$. Throughout this section, we will employ the following blanket assumption on \eqref{eqn: 2SP}.
\begin{assumption}\label{assum: 2SP basic assumption} 
\noindent The following conditions are in effect for \textnormal{\eqref{eqn: 2SP}}:
\begin{itemize}
\item[\textbf{(C1)}] For a.e. $\xi \in \Xi$, the function $\widehat{F}(\cdot,y,\xi) \colon \mathbb{R}^n \rightarrow \mathbb{R}$ is differentiable for every $y \in \mathcal{Y}$ and $\nabla_x \hat{F}(\cdot,\cdot,\xi)$ is continuous on $\mathbb{R}^n \times \mathcal{Y}$, for a.e. $\xi \in \Xi$. Moreover, for all $(x,y) \in \mathbb{R}^n \times \mathbb{R}^m$, the function $\widehat{F}(x,y,\cdot)$ is Borel measurable;
\item[\textbf{(C2)}] The set $\mathcal{X}\subset \mathbb{R}^n$ is convex and compact and the multifunction $\mathcal{Y}\colon \Xi \rightrightarrows \mathbb{R}^m$ is compact-valued and Borel measurable;
\item[\textbf{(C3)}] We can draw i.i.d. samples from the law of $\xi$;
\item[\textbf{(C4)}] For a.e. $\xi \in \Xi$, the function $F(x,
\xi)$ satisfies condition \textnormal{\textbf{(A1)}} of Assumption \textnormal{\ref{assum: basic assumption}}.
\end{itemize}
\end{assumption}
\par Let us now briefly consider the conditions imposed in Assumption \ref{assum: 2SP basic assumption}. Specifically, the only condition that requires verification is \textbf{(C4)}. We proceed to argue that this condition is indeed minimal. Specifically, for a.e. $\xi \in \Xi$, and by using conditions \textbf{(C1)}--\textbf{(C2)} of Assumption \ref{assum: 2SP basic assumption}, we may employ Danksin's theorem (e.g., see \cite[Theorem 9.26]{ShapiroLecturesStochProg3rd}), which implies that $F(\cdot,\xi)$ is $L(\xi)-$Lipschitz continuous on $\mathcal{X}$ (since $\mathcal{X}$ is assumed to be compact). In other words, condition \textbf{(C4)} merely enforces that $\mathbb{E}\{L^2(\xi)\} \leq G^2$, for some $G > 0$.\\
\paragraph{\textbf{Applying Algorithm \ref{Algorithm: Z-iProxSG} to \eqref{eqn: 2SP}}} 
\par Next, we discuss the compatibility of \eqref{eqn: 2SP} (alongside Assumption \ref{assum: 2SP basic assumption}) with the developments in Section \ref{sec: inexact zeroth-order method}, and in particular with the oracle definition (cf. Definition \ref{def: inexact noisy oracle} and its associated Assumption \ref{assum: second assumption}). Let us begin by noting that condition \textbf{(B2)} of Assumption \ref{assum: second assumption} is already satisfied since we have assumed that $\mathcal{X}$ is a convex and compact set. We next discuss the plausibility of condition \textbf{(B1)} instead, while discussing the compatibility of the inexact noisy oracle given in Definition \ref{def: inexact noisy oracle} with \eqref{eqn: 2SP}.
\par We start by noting that Assumption \ref{assum: 2SP basic assumption} does not enforce convexity of the second-stage problem, i.e., of $\min_{y \in \mathcal{Y}(\xi)} F(x,y,\xi)$, given some $(x,\xi) \in \mathbb{R}^n \times \Xi$. Nonetheless, the oracle definition implicitly assumes that we can consistently employ some algorithm (e.g., a numerical method) that is able to find, for any $x \in \mathcal{X}$ and a.e. $\xi \in \Xi$, a solution $\tilde{y}(x,\xi) \in \mathcal{Y}(\xi)$ such that
\[ F(x,\tilde{y}(x,\xi),\xi) - \min_{y \in \mathcal{Y}(\xi)} F(x,y,\xi) = \delta(x,\xi) \leq \tilde{\delta},\]
\noindent where, in this case, the absolute value is obsolete (by the definition of problem \eqref{eqn: 2SP}). In what follows, we argue that the proposed algorithm provides a general-purpose solution method for two-stage stochastic programming problems that go far beyond what has already been considered in the literature. To that end, we discuss Assumption \ref{assum: 2SP basic assumption} by separating cases, first considering lower-level convexity, and the discussing the general nonconvex lower-level case.
\begin{enumerate}
\item The first case that is naturally covered in our setup is the case where $\mathcal{Y}(\xi)$ is a convex set, and for any $x \in \mathcal{X}$ and a.e. $\xi \in \Xi$, $\hat{F}(x,\cdot,\xi)$ is a convex function. Let us note that this does not imply that $\eqref{eqn: 2SP}$ is a convex problem, since we have not enforced convexity on $\hat{F}(\cdot,y,\xi)$. Additionally, and unlike most approaches in the literature, the methodology works without imposing any regularity conditions on the second-stage problem (such as, e.g., Slater's or Robinson's constraint qualifications), neither uniqueness of the second-stage optimal solution for a fixed pair $(x,\xi)$. In the latter case, i.e., under the assumption of uniqueness of the second-stage problem's optimal solution (which, for example, follows under the assumption of strong convexity of $\hat{F}(x,\cdot,\xi)$), one could invoke Danksin's theorem (again, see \cite[Theorem 9.26]{ShapiroLecturesStochProg3rd}) to show that $F(\cdot,\xi)$ is differentiable (in which case, one could attempt to solve \eqref{eqn: 2SP} by utilizing stochastic hypergradient descent; e.g., see the developments in \cite{Hashmi_etal_ICASSP}, which focus on two-stage programming problems arising in wireless communication systems). In the general framework of stochastic bilevel optimization, which subsumes two-stage stochastic programming, hypergradient descent schemes that rely on lower-level strong convexity have been well-studied. We refer the reader to \cite{JMLR:Chen_etal,GrazziPS21}, and the reference therein, for additional details.
\par In this general setting, we can assume that the lower level problem can be solved to any accuracy, thus making our oracle as accurate as needed. Thus, following our discussion in Section \ref{subsec: convergence analysis}, we may readily enforce that $\tilde{\delta} = \mathcal{O}(\mu/\sqrt{n})$ (by making use of an appropriate convex numerical optimization solver), thus retrieving the convergence rate achieved in \cite[Theorem 3.2]{NEURIPS2022_Linetal}. Obviously, if condition \textbf{(B1)} of Assumption \ref{assum: second assumption} is not satisfied, it instead suffices to enforce that $\tilde{\delta} = \mathcal{O}(\epsilon^2 \mu/n)$ (since condition \textbf{(B2)} of Assumption \ref{assum: second assumption} is readily satisfied) to obtain the same rate.
\par Finally, let us observe that condition \textbf{(B1)} of Assumption \ref{assum: second assumption} is very natural in this case. Indeed, as already discussed in Remark \ref{remark: inexact oracle}, we can call a numerical optimization solver for the lower level problem (i.e., for $\min_{y \in \mathcal{Y}(\xi)} \hat{F}(x,y,\xi)$) and enforce that it returns a solution of prescribed accuracy, for any $x \in \mathcal{X}$. Thus, the discussion in Remark \ref{remark: inexact oracle} readily applies in this context.
\item In general, the conditions given in Assumption \ref{assum: 2SP basic assumption} do not exclude the case where $\hat{F}(x,\cdot,\xi)$ is nonconvex. In this case, the oracle given in Definition \ref{def: inexact noisy oracle} is still consistent and general enough. Indeed, we do not specify the magnitude of $\tilde{\delta}$ in this definition (which refers to the upper bound on the oracle error). Thus, the proposed algorithm works as intended also in this case. The difference to the convex lower-level case is that we can no longer control the magnitude of $\tilde{\delta}$ to an arbitrary degree (unless further structure is imposed to the lower-level problem). Thus, we cannot expect to retrieve the same convergence rates as those derived in \cite[Theorem 3.2]{NEURIPS2022_Linetal}, and would instead have to settle for an approximately stationary point, with the approximation accuracy directly dependent on the oracle error bound $\tilde{\delta}$.
\par Concerning condition \textbf{(B1)} of Assumption \ref{assum: second assumption}, the situation is less clear compared with the convex lower-level case. Specifically, the discussion given in Remark \ref{remark: inexact oracle} is no longer necessarily applicable. Instead, what this condition implies is that the lower-level problem, while nonconvex, is ``equally hard", irrespectively of $x \in \mathcal{X}$. In other words, this condition implicitly assumes that the lower-level problem can be solved to a similar accuracy, irrespectively of the outer-level parameter vector $x$. While this is not a strong assumption, it is not readily verifiable; for that reason, Algorithm \ref{Algorithm: Z-iProxSG} was analyzed also under condition \textbf{(B2)} of Assumption \ref{assum: second assumption}, which automatically holds in this case.
\end{enumerate}

\paragraph{\textbf{Comparison with alternative solution methods}} 
\par Let us now compare the proposed methodology with alternative optimization schemes that have been devised in the literature to solve problems of the form of \eqref{eqn: 2SP}. There are currently two classes of methods suitable for the solution of \eqref{eqn: 2SP} in the available literature, namely, stochastic successive convex approximation (SSCA) optimization and stochastic hypergradient descent schemes. 
\par Specifically, there is a long line of works focusing on SSCA-type methods for the solution of nonconvex two-stage stochastic programs studied herein, which were heavily utilized in the context of optimization over wireless communication networks and resource allocation (e.g., see \cite{Hong_SSCO,Scutari_SSCO,Yang_SSCO} and the references therein). Such methods, which are typically classified as ``two-timescale schemes", rely on successive convex surrogates and approximation of the problem statistics during the optimization process, which incurs high computational costs as well as unrealistic assumptions for their theoretical grounding (which does not include non-asymptotic guarantees).
\par Many of the drawbacks of SSCA schemes were later addressed in a line of work focused on stochastic hypergradient descent schemes (see \cite{Hashmi_etal_ICASSP,Hashmi_etal_IEEETran,Pougk_etal_IRSinexact}), which avoid the use of any problem statistics (enabling the online execution of the associated algorithms) while also providing much stronger theoretical guarantees compared with SSCA-type methods. 
\par Specifically, the work in \cite{Hashmi_etal_IEEETran} provided a detailed theoretical analysis of these stochastic hypergradient schemes under fairly general assumptions, under the condition that the second-stage problem is solved exactly. This was later relaxed in \cite{Pougk_etal_IRSinexact}, which allowed for inexact evaluations of the objective function of \eqref{eqn: 2SP}. Nonetheless, both approaches require weak convexity and differentiability of $F(\cdot,\xi)$, which in turn can only be established under some strong assumptions on the second-stage problem (i.e., the minimization with respect to $y$). Most notably, the strong second-order sufficient conditions at each optimal solution of the second-stage problem stand out, since they imply a local solution uniqueness property for the second-stage problem, which is known to fail in many circumstances in nonconvex optimization (e.g., see \cite{ShapiroPerturbationAnalysis}). \par Additionally, while \cite{Pougk_etal_IRSinexact} allows for oracle errors, there is still a requirement that the distance between the retrieved approximate solution to the second-stage problem (returned by the oracle) is ``close" to some optimal solution. In this work, we instead only require that the objective values of these two points are close, which is much more consistent in the context of nonconvex optimization. Indeed, in order to guarantee that this ``closeness" of the oracle point to some optimal solution, required in \cite{Pougk_etal_IRSinexact}, is satisfied, the authors had to restrict the class of functions $\hat{F}(\cdot,\cdot,\xi)$ to those that are real-analytic. While this class is fairly rich, it is significantly more limited compared with the functions included in this work, under Assumption \ref{assum: 2SP basic assumption}. 
\par As expected, assuming that the conditions required by stochastic hypergradient schemes are satisfied for \eqref{eqn: 2SP}, one may obtain better rates compared with those derived herein, and the corresponding algorithms require a single oracle evaluation at each outer stochastic hypergradient iteration. Nonetheless, under an additional oracle evaluation, we showcase that the proposed approach given in Algorithm \ref{Algorithm: Z-iProxSG} can operate under significantly more general assumptions, thus substantially advancing the known capabilities of numerical optimization for nonconvex two-stage stochastic programming.

\subsection{Stochastic minimax optimization} \label{subsec: minimax}
\par Next, we consider general stochastic minimax optimization problems. To that end, we will separate two cases, which typically require distinct solution methods and are naturally applied in different contexts. Specifically, we first consider the case of minimax stochastic optimization problems in which the ``adversary" has complete access to instantaneous information, while in the second case, we will assume that the ``adversary" only has access to ergodic information. A typical application of great importance for the former formulation is that of building deep learning models that are resistant to adversarial attacks (e.g., see \cite{madry2018towards}), while the latter formulation typically appears in applications involving generative adversarial networks, distributionally robust optimization or robust training of neural networks, among others (e.g., see \cite{goodfellow2014generative,Rahimian2019DistributionallyRO,awasthi2021certifying}).
\subsubsection{Adversary with instantaneous information}
\par We first consider stochastic minimax optimization problems of the following form:
\begin{equation} \label{eqn: minimax instantaneous} \tag{MM-I}
\min_{x \in \mathcal{X}} \mathbb{E}_{\xi}\left\{\max_{y \in \mathcal{Y}(\xi)} \hat{F}(x,y,\xi) \right\},
\end{equation}
\noindent where, as in the two-stage programming case, we assume that the feasible set of the adversarial variable $y$ is independent of $x$ but may depend on the random vector $\xi$ (noting that typical applications assume that $\mathcal{Y}$ is also independent of $\xi$; e.g., see \cite{madry2018towards}). 
\par Once again, in keeping with the notation of \eqref{eqn: main problem}, we let $F(x,\xi) \triangleq \max_{y \in \mathcal{Y}(\xi)} \hat{F}(x,y,\xi)$. Furthermore, to ensure that problem \eqref{eqn: minimax instantaneous} is well-defined, we will again implicitly make the minimal assumption that $\widehat{F}(x,y^*(x,\xi(\cdot)),\xi(\cdot)) \in \mathcal{L}_1(\Omega,\mathscr{F},P;\mathbb{R})$ for any measurable selection $y^*(x,\xi(\cdot)) \in \arg\max_{y \in \mathcal{Y}(\xi(\cdot))} \widehat{F}(x,y,\xi(\cdot))$.
\par Let us note the similarity between problem \eqref{eqn: 2SP} and \eqref{eqn: minimax instantaneous}. Indeed, the only difference between the two formulations is the maximization with respect to $y$ in the latter case, compared with the minimization present in the former case. In light of this, our entire discussion given in Section \ref{subsec: 2SP} applies readily also in this case, and thus Algorithm \ref{Algorithm: Z-iProxSG} can be immediately utilized to solve \eqref{eqn: minimax instantaneous}, while being theoretically supported under Assumption \ref{assum: 2SP basic assumption}. At this point, it is important to note that the presence of a maximization (instead of a minimization) term in the objective function of \eqref{eqn: minimax instantaneous} does not change anything structurally important \textit{in the context of Assumption \ref{assum: 2SP basic assumption}}. Indeed, as we did in Section \ref{subsec: 2SP}, we can once again apply Danskin's theorem to show that conditions \textbf{(C1)}--\textbf{(C2)} imply $L(\xi)-$Lipschitz continuity of $F(\cdot,\xi)$, for a.e. $\xi \in \Xi$, which in turn yields that condition \textbf{(C4)} is merely an assumption on the boundedness of second-moment of the associated Lipschitz constant random function $L(\cdot)$. 
\par This example, when paired with the example presented in Section \ref{subsec: 2SP}, immediately showcases the power of the scheme presented in Algorithm \ref{Algorithm: Z-iProxSG}. Indeed, the same algorithmic strategy can be readily employed to solve two distinct optimization problems that are traditionally challenging, under the same minimal conditions collected in Assumption \ref{assum: 2SP basic assumption}. 
\par What is especially interesting in this case is the juxtaposition of the proposed methodology (applied in the context of instantaneous stochastic minimax optimization) and the methodology presented in \cite{madry2018towards}, which lacks any serious theoretical support (let alone under the minimal set of assumptions laid out herein). Nonetheless, as we have also stated in the introduction, an adaptation of the stochastic hypergradient descent scheme proposed in \cite{Pougk_etal_IRSinexact} could potentially be theoretically supported in this case under certain regularity and structural conditions (albeit stronger compared with those laid out in Assumption \ref{assum: 2SP basic assumption}). This is left as a future research direction, open to further consideration.

\subsubsection{Adversary with ergodic information}
\par Next, we consider stochastic (nonconvex-nonconcave) minimax optimization problems of the form:
\begin{equation} \label{eqn: minimax ergodic} \tag{MM-E} 
\min_{x \in \mathcal{X}} \max_{y \in \mathcal{Y}} \mathbb{E}_{\xi}\left\{\hat{F}(x,y,\xi) \right\} \triangleq \hat{f}(x,y),
\end{equation}
\noindent where the feasible set for the adversarial variable, i.e. $\mathcal{Y}$, is now independent of both $x$ and $\xi$, and the adversary is assumed to only have access to ergodic information. Similar to the previous two applications, we let $F(x,\xi) \triangleq \hat{F}(x,y^*(x),\xi)$, where $y^*(x) \in \arg\max_{y \in \mathcal{Y}} \{\mathbb{E}_{\xi}\{\hat{F}(x,y,\xi)\}\}$ is some selection. Let us observe that, for any two selections $y_1^*(x),\ y_2^*(x)$, the objective function value of problem \eqref{eqn: main problem} is the same, i.e. $\mathbb{E}_{\xi}\{F(x,y_1^*(x),\xi)\}  = \mathbb{E}_{\xi}\{F(x,y_2^*(x),\xi)\}$. This detail will be important in order to show that problem \eqref{eqn: minimax ergodic} can be cast in the form of \eqref{eqn: main problem} and satisfy the conditions of Assumption \ref{assum: basic assumption} under mild assumptions.
\par Indeed, let us now briefly discuss Assumption \ref{assum: basic assumption} in the context of the following problem:
\[ \min_{x \in \mathcal{X}} f(x) = \mathbb{E}_{\xi}\{\hat{F}(x,y^*(x),\xi)\},\]
\noindent where $y^*(x) \in \arg\max_{y \in \mathcal{Y}} \{\mathbb{E}_{\xi}\{\hat{F}(x,y,\xi)\}\}$ is an arbitrary selection. We note that under the assumption of compactness of $\mathcal{Y}$, the differentiability of $\hat{f}(\cdot,y)$ for any $y \in \mathcal{Y}$, and the continuity of $\nabla_x \hat{f}(x,y)$ on $\mathbb{R}^n \times \mathcal{Y}$, we would obtain that $f(\cdot)$ is Lipschitz continuous on $\mathcal{X}$. The requirement of Assumption \ref{assum: basic assumption} is slightly stronger, in that it enforces Lipschitz continuity of $F(\cdot,y^*(\cdot),\xi)$ rather than of $f$ (alongside the second-moment condition of the associated Lipschitz continuity constant).  For example, the former could be guaranteed under the following assumptions on $\hat{F}$, without the requirement that $\mathcal{Y}$ is compact:
\begin{itemize}
\item $\hat{F}(x,y,\xi)$ is twice-differentiable with respect to $y$ for all $x \in \mathcal{X}$ and a.e. $\xi \in \Xi$ and the Hessian (w.r.t. $y$) is continuous jointly in $(x,y)$;
\item $\hat{F}(x,y,\xi)$ is Lipschitz continuous with respect to $x$, uniformly in $y$, for a.e. $\xi \in \Xi$.
\end{itemize}
Under these assumptions, we may utilize Robinson's implicit function theorem (e.g., see \cite[Theorem 2B.1]{dontchev2009implicit}) to show that, for all $x \in \mathcal{X}$, there exists a locally Lipschitz continuous selection $y^*(x)$, which in turn implies that $F(x,y^*(x),\xi)$ is locally Lipschitz continuous. Lipschitz continuity is then retrieved by assuming that $\mathcal{X}$ is compact. Note that the fact that the single-valued locally Lipschitz localization $y^*(x)$ cannot necessarily be found in practice is not a problem for the proposed methodology. Indeed, since $f(x)$ has the same value for all selections, and since our algorithm operates under the assumption that $F(x,\xi)$ can only be evaluated inexactly, we may define $F$ using any measurable selection $y^*(\cdot)$; in turn, this can ensure that $F$ satisfies the conditions of Assumption \ref{assum: basic assumption} under mild assumptions.
\par Once again, we see that the proposed algorithmic framework is readily applicable to problems of the form of \eqref{eqn: minimax ergodic}, and its nonasymptotic convergence guarantees hold under less restrictive assumptions compared with alternative approaches provided in the literature (e.g.,  see the developments in \cite{pmlr-v195-daskalakis23b,diakonikolas2021efficient,grimmer2022landscape} and the references therein). We note, however, that the case in which the lower-level (maximization) problem is concave is typically best handled using stochastic gradient descent-ascent schemes (e.g., see the developments in \cite{JMLR:v26:22-0863} and the references therein), assuming that the sample gradients of $\hat{F}$ can be readily computed (which is not a requirement for the method proposed herein). 
\par Overall, we observe that the proposed approach is highly versatile and enables the approximate solution of intractable optimization instances under very general assumption that are out of reach for currently available gradient-based methodologies.

\subsection{Additional applications}
\par Let us observe that while Sections \ref{subsec: 2SP}--\ref{subsec: minimax} focus on cases where $F(\cdot,\xi)$ represents the value function of some optimization problem, the proposed algorithm is applicable in a plethora of other settings in which the presence of inexact oracles for the evaluation of the objective function of \eqref{eqn: main problem} remains crucial. Indeed, a natural example includes cases in which the evaluation of $F(\cdot,\xi)$ requires the utilization of some numerical simulation of a real-world process (e.g., involving the solution of discretized partial differential equations, among many other examples). In this case, the evaluation oracles remain inexact and problem \eqref{eqn: main problem} enables one to solve general stochastic parametric problems under the minimal assumption of Lipschitz continuity. Applications of this form appear in several real-world domains, and are typically classified as hyperparameter tuning problems (e.g., see \cite[Section 4.2]{Pougk_SISC} for an example problem in the context of hyperparameter tuning of algorithmic parameters). For simplicity of exposition, we refer the reader to \cite[Chapter 4]{powell2022reinforcement} for a detailed discussion on several application instances of this form.
\section{Conclusions} \label{sec: conclusions}
\par In this work, we derive a zeroth-order method suitable for the solution of general nonsmooth and nonconvex stochastic composite optimization problems in which the real-valued part of the objective is Lipschitz continuous while the extended-valued one is closed, proper, and convex. The algorithm is shown to converge, non-asymptotically, close to a stationary point under minimal assumptions, where near-stationarity is controlled using a novel optimality measure proposed herein (generalizing notions that are currently available in the literature). Importantly, the algorithm is able to operate under general stochastic oracles, providing inexact and biased evaluations of the stochastic objective function. 
\par In light of the generality of the proposed algorithm, we showcase its ability of handling (in a theoretically supported manner) large classes of two-stage stochastic programming as well as nonconvex-nonconcave stochastic minimax optimization problems, in regimes that are out-of-reach of alternative optimization methods that are currently available in the literature. Specifically, we demonstrate the versatility of the proposed methodology by juxtaposing the assumptions required to establish its non-asymptotic ergodic convergence in several challenging applications against the assumptions required by alternative state-of-the-art approaches appearing in the literature.

\vskip 6mm
\noindent{\bf Acknowledgments}

\noindent Dionysis Kalogerias was supported by the  National Natural Science Foundation  under Grant No. 2242215.

\bibliographystyle{unsrt}
\bibliography{references}

@book{ShapiroPerturbationAnalysis,
    author = "Bonnans, J. F. and Shapiro, A.",
    title = "{P}erturbation {A}nalysis of {O}ptimization {P}roblems",
    series = "Springer Series in Operations Research and Financial Engineering",
    publisher = "Springer New York, NY",
    year = "2000"
}

@book{ShapiroLecturesStochProg3rd,
author = {Shapiro, A. and Dentcheva, D. and Ruszczynski, A.},
title = {{L}ectures on {S}tochastic {P}rogramming: {M}odeling and {T}heory, {T}hird {E}dition},
publisher = {Society for Industrial and Applied Mathematics},
year = {2021},
address = {Philadelphia, PA},
}

@book{Clarke,
author = {Clarke, F. H.},
title = {{O}ptimization and {N}onsmooth {A}nalysis},
publisher = {Society for Industrial and Applied Mathematics},
year = {1990}
}

@book{Rockafellar2009VarAn,
author = {Rockafellar, R. T. and Wets, R. J-B},
publisher = {Springer Science \& Business Media},
title = {Variational Analysis},
doi = {10.1007/978-3-642-02431-3},
address = {Heidelberg, DE},
volume = {317},
year = {2004}
}

@InProceedings{diakonikolas2021efficient,
  title     = {{E}fficient methods for structured nonconvex-nonconcave min-max optimization},
  author    = {Diakonikolas, J. and Daskalakis, C. and Jordan, M. I.},
  booktitle = {Proceedings of the 24th International Conference on Artificial Intelligence and Statistics (AISTATS)},
  series    = {Proceedings of Machine Learning Research},
  volume    = {130},
  pages     = {2746--2754},
  year      = {2021}
}

@article{grimmer2022landscape,
  title     = {{T}he landscape of the proximal point method for nonconvex–nonconcave minimax optimization},
  author    = {Grimmer, B. and Lu, H. and Worah, P. and Mirrokni, V.},
  journal   = {Mathematical Programming},
  year      = {2022},
  volume    = {191},
  number    = {1, Série A},
  pages     = {1--35}
}

@InProceedings{pmlr-v195-daskalakis23b,
  title     = {{S}tay-on-the-ridge: {G}uaranteed convergence to local minimax equilibrium in nonconvex-nonconcave games},
  author    = {Daskalakis, C. and Golowich, N. and Skoulakis, S. and Zampetakis, E.},
  booktitle = {Proceedings of the 36th Conference on Learning Theory (COLT)},
  series    = {Proceedings of Machine Learning Research},
  volume    = {195},
  pages     = {5146--5198},
  year      = {2023}
}

@book{powell2022reinforcement,
  author    = {Powell, W. B.},
  title     = {{R}einforcement {L}earning and {S}tochastic {O}ptimization: {A} {U}nified {F}ramework for {S}equential {D}ecisions},
  year      = {2022},
  publisher = {John Wiley \& Sons},
  address   = {Hoboken, NJ}
}

@book{dontchev2009implicit,
  author    = {Dontchev, A. L. and Rockafellar, R. T.},
  title     = {{I}mplicit {F}unctions and {S}olution {M}appings},
  series    = {Springer Monographs in Mathematics},
  year      = {2009},
  publisher = {Springer, Dordrecht}}

@article{Goldstein,
    author = "Goldstein, A. A.",
    title = "{O}ptimization of {L}ipschitz continuous functions",
    journal = "Mathematical Programming",
    year = "1977",
    volume = "13",
    pages = "14--22"
}

@article{DavisDrus_SIOPT,
author = {Davis, D. and Drusvyatskiy, D.},
title = {{S}tochastic model-based minimization of weakly convex functions},
journal = {SIAM Journal on Optimization},
volume = {29},
number = {1},
pages = {207--239},
year = {2019}
}

@inproceedings{NEURIPS2022_2c8d9636,
 author = {Davis, Damek and Drusvyatskiy, Dmitriy and Lee, Yin Tat and Padmanabhan, Swati and Ye, Guanghao},
 booktitle = {Advances in Neural Information Processing Systems},
 editor = {S. Koyejo and S. Mohamed and A. Agarwal and D. Belgrave and K. Cho and A. Oh},
 pages = {6692--6703},
 publisher = {Curran Associates, Inc.},
 title = {{A} gradient sampling method with complexity guarantees for {L}ipschitz functions in high and low dimensions},
 volume = {35},
 year = {2022}
}

@inproceedings{NEURIPS2022_Linetal,
 author = {Lin, T. and Zheng, Z. and Jordan, M.},
 booktitle = {Advances in Neural Information Processing Systems},
 pages = {26160--26175},
 title = {{G}radient-free methods for deterministic and stochastic nonsmooth nonconvex optimization},
 volume = {35},
 year = {2022}
}

@article{JMLR:KornowskiShamir,
  author  = {Kornowski, G. and Shamir, O.},
  title   = {{A}n algorithm with optimal dimension-dependence for zero-order nonsmooth nonconvex stochastic optimization},
  journal = {Journal of Machine Learning Research},
  year    = {2024},
  volume  = {25},
  number  = {122},
  pages   = {1--14}
}

@inproceedings{Liu_etal_ICML,
author = {Liu, Zhuanghua and Chen, Cheng and Luo, Luo and Low, Bryan Kian Hsiang},
title = {{Z}eroth-order methods for constrained nonconvex nonsmooth stochastic optimization},
year = {2024},
publisher = {JMLR.org},
booktitle = {Proceedings of the 41st International Conference on Machine Learning},
articleno = {1243},
numpages = {31},
series = {ICML'24}
}

@inproceedings{Chen_etal_ICML,
author = {Chen, Lesi and Xu, Jing and Luo, Luo},
title = {{F}aster gradient-free algorithms for nonsmooth nonconvex stochastic optimization},
year = {2023},
publisher = {JMLR.org},
booktitle = {Proceedings of the 40th International Conference on Machine Learning},
articleno = {205},
numpages = {15},
series = {ICML'23}
}

@book{nemirovskiyudin1983,
    author = "Nemirovsky, A. and Yudin, D.",
    title = "{P}roblem {C}omplexity and {M}ethod {E}fficiency in {O}ptimization",
    publisher = "John Wiley and Sons, New York",
    year = "1983"
}

@InProceedings{pmlr-v195-jordan23a,
  title = 	 {{D}eterministic nonsmooth nonconvex optimization},
  author =       {Jordan, M. and Kornowski, G. and Lin, T. and Shamir, O. and Zampetakis, M.},
  booktitle = 	 {Proceedings of Thirty Sixth Conference on Learning Theory},
  pages = 	 {4570--4597},
  year = 	 {2023},
  editor = 	 {Neu, Gergely and Rosasco, Lorenzo},
  volume = 	 {195},
  series = 	 {Proceedings of Machine Learning Research},
  publisher =    {PMLR}
}

@inproceedings{Cutkosky_ICML,
author = {Cutkosky, Ashok and Mehta, Harsh and Orabona, Francesco},
title = {{O}ptimal, stochastic, non-smooth, non-convex optimization through online-to-non-convex conversion},
year = {2023},
publisher = {JMLR.org},
booktitle = {Proceedings of the 40th International Conference on Machine Learning},
articleno = {266},
numpages = {28},
location = {Honolulu, Hawaii, USA},
series = {ICML'23}
}

@article{ATTOUCH1993,
title = {{A}pproximation and regularization of arbitrary functions in {H}ilbert spaces by the {L}asry-{L}ions method},
journal = {Annales de l'Institut Henri Poincaré C, Analyse non linéaire},
volume = {10},
number = {3},
pages = {289--312},
year = {1993},
author = {Attouch, H. and Aze, D.}
}

@article{Vial_WeakConvexity,
 author = {Vial, J.-P.},
 journal = {Mathematics of Operations Research},
 number = {2},
 pages = {231--259},
 publisher = {INFORMS},
 title = {{S}trong and weak convexity of sets and functions},
 volume = {8},
 year = {1983}
}

@article{Pougk_SISC,
author = {Pougkakiotis, S. and Kalogerias, D.},
title = {{A} zeroth-order proximal stochastic gradient method for weakly convex stochastic optimization},
journal = {SIAM Journal on Scientific Computing},
volume = {45},
number = {5},
pages = {A2679--A2702},
year = {2023}
}

@inproceedings{Flaxman_etal,
author = {Flaxman, A. D. and Kalai, A. T. and McMahan, H. B.},
title = {{O}nline convex optimization in the bandit setting: {G}radient descent without a gradient},
year = {2005},
publisher = {Society for Industrial and Applied Mathematics},
booktitle = {Proceedings of the Sixteenth Annual ACM-SIAM Symposium on Discrete Algorithms},
pages = {385–-394},
series = {SODA '05}
}

@article{JMLR:Chen_etal,
  author  = {X. Chen and T. Xiao and K. Balasubramanian},
  title   = {{O}ptimal algorithms for stochastic bilevel optimization under relaxed smoothness conditions},
  journal = {Journal of Machine Learning Research},
  year    = {2024},
  volume  = {25},
  number  = {151},
  pages   = {1--51}
}

@inproceedings{GrazziPS21,
  title = {{C}onvergence properties of stochastic hypergradients},
  author = {Riccardo G. and Massimiliano P. and Saverio S.},
  year = {2021},
  pages = {3826--3834},
  booktitle = {The 24th International Conference on Artificial Intelligence and Statistics, AISTATS 2021},
  volume = {130},
  series = {Proceedings of Machine Learning Research},
  publisher = {PMLR},
}

@article{Pougk_etal_IRSinexact,
    author = "Pougkakiotis, S. and Hashmi, H. and Kalogerias, D.",
    title = "{D}ata-driven learning of two-stage beamformers in passive {IRS}-assisted systems with inexact oracles",
    journal = "	arXiv:2410.24154",
    year = "2024"
}

@inproceedings{madry2018towards,
  title={{T}owards deep learning models resistant to adversarial attacks},
  author={Madry, Aleksander and Makelov, Aleksandar and Schmidt, Ludwig and Tsipras, Dimitris and Vladu, Adrian},
  booktitle={International Conference on Learning Representations (ICLR)},
  year={2018}
}

@ARTICLE{Hashmi_etal_IEEETran,
  author={Hashmi, Hassaan and Pougkakiotis, Spyridon and Kalogerias, Dionysis},
  journal={IEEE Transactions on Signal Processing}, 
  title={{M}odel-free learning of two-stage beamformers for passive {IRS}-aided network design}, 
  year={2024},
  volume={72},
  number={},
  pages={652--669}}

@book{nesterov2018lectures,
  title     = {{L}ectures on {C}onvex {O}ptimization},
  author    = {Nesterov, Yurii},
  volume    = {137},
  year      = {2018},
  publisher = {Springer},
  series    = {Springer Optimization and Its Applications}
}

@article{nesterov2017random,
  author    = {Yurii E. Nesterov and Vladimir G. Spokoiny},
  title     = {{R}andom gradient-free minimization of convex functions},
  journal   = {Foundations of Computational Mathematics},
  year      = {2017},
  volume    = {17},
  number    = {2},
  pages     = {527--566},
}

@article{Rahimian2019DistributionallyRO,
  title={{D}istributionally robust optimization: {A} review},
  author={Rahimian, H. and Mehrotra, S.},
  journal={ArXiv},
  year={2019},
  volume={abs/1908.05659}
}

@ARTICLE{Scutari_SSCO,
  author={Scutari, G. and Facchinei, F. and Song, P. and Palomar, D. P. and Pang, J.-S.},
  journal={IEEE Transactions on Signal Processing}, 
  title={{D}ecomposition by partial linearization: {P}arallel optimization of multi-agent systems}, 
  year={2014},
  volume={62},
  number={3},
  pages={641--656}}

@ARTICLE{Hong_SSCO,
  author={Hong, M. and Li, Q. and Liu, Y.-F.},
  journal={IEEE Transactions on Wireless Communications}, 
  title={{D}ecomposition by successive convex approximation: {A} unifying approach for linear transceiver design in heterogeneous networks}, 
  year={2016},
  volume={15},
  number={2},
  pages={1377--1392}}

@ARTICLE{Yang_SSCO,
  author={Yang, Y. and Scutari, G. and Palomar, D. P. and Pesavento, M.},
  journal={IEEE Transactions on Signal Processing}, 
  title={{A} parallel decomposition method for nonconvex stochastic multi-agent optimization problems}, 
  year={2016},
  volume={64},
  number={11},
  pages={2949--2964},
  }

@inproceedings{10.5555/2888116.2888133,
author = {Tamar, A. and Glassner, Y. and Mannor, S.},
title = {{O}ptimizing the {CV}a{R} via sampling},
year = {2015},
publisher = {AAAI Press},
booktitle = {Proceedings of the Twenty-Ninth AAAI Conference on Artificial Intelligence},
pages = {2993-–2999},
series = {AAAI'15}
}

@InProceedings{pmlr-v238-emmanouilidis24a,
  title = 	 {{S}tochastic extragradient with random reshuffling: {I}mproved convergence for variational inequalities},
  author =       {Emmanouilidis, K. and Vidal, R. and Loizou, N.},
  booktitle = 	 {Proceedings of The 27th International Conference on Artificial Intelligence and Statistics},
  pages = 	 {3682--3690},
  year = 	 {2024},
  volume = 	 {238},
  series = 	 {Proceedings of Machine Learning Research},
  publisher =    {PMLR}
}

@article{JMLR:v26:22-0863,
  author  = {Lin, T. and Jin, C. and Jordan, M. I.},
  title   = {{T}wo-timescale gradient descent ascent algorithms for nonconvex minimax optimization},
  journal = {Journal of Machine Learning Research},
  year    = {2025},
  volume  = {26},
  number  = {11},
  pages   = {1--45}
}

@article{duchi2012randomized,
  author  = {John C. Duchi and Peter L. Bartlett and Martin J. Wainwright},
  title   = {{R}andomized smoothing for stochastic optimization},
  journal = {SIAM Journal on Optimization},
  year    = {2012},
  volume  = {22},
  number  = {2},
  pages   = {674--701},
}

@book{Cesa-Bianchi_Lugosi_2006,
place={Cambridge}, title={{P}rediction, {L}earning, and {G}ames}, publisher={Cambridge University Press}, author={Cesa-Bianchi, N. and Lugosi, G.}, year={2006}}

@inproceedings{goodfellow2014generative,
  title     = {{G}enerative adversarial nets},
  author    = {Goodfellow, I. and Pouget-Abadie, J. and Mirza, M. and Xu, B. and Warde-Farley, D. and Ozair, S. and Courville, A. and Bengio, Y.},
  booktitle = {Advances in Neural Information Processing Systems (NeurIPS)},
  year      = {2014},
  pages     = {2672--2680}
}

@inproceedings{awasthi2021certifying,
  title     = {{C}ertifying some distributional
robustness with principled adversarial training},
  author    = {Sinha, A. and Namkoong, H. and Duchi, J. C.},
  booktitle = {International Conference on Learning Representations},
  year      = {2018}
}

@article{GrimmerOptLetters,
    author = "Grimmer, B. and Jia, Z. ",
    title = "{G}oldstein stationarity in Lipschitz constrained optimization",
    journal = "Optimization Letters",
    volume = "19",
    pages = "425--435",
    year = "2025"
}

@INPROCEEDINGS{Hashmi_etal_ICASSP,
  author={Hashmi, H. and Pougkakiotis, S. and Kalogerias, D.},
  booktitle={2023 IEEE International Conference on Acoustics, Speech and Signal Processing (ICASSP)}, 
  title={{M}odel-free learning of optimal beamformers for passive {IRS}-assisted sumrate maximization}, 
  year={2023},
  pages={1--5}}

@article{JMLR:Shamir,
  author  = {Shamir, O.},
  title   = {{A}n optimal algorithm for bandit and zero-order convex optimization with two-point feedback},
  journal = {Journal of Machine Learning Research},
  year    = {2017},
  volume  = {18},
  number  = {52},
  pages   = {1--11}
}

@InProceedings{pmlr:Zhang_etal,
  title = 	 {{C}omplexity of finding stationary points of nonconvex nonsmooth functions},
  author =       {Zhang, J. and Lin, H. and Jegelka, S. and Sra, S. and Jadbabaie, A.},
  booktitle = 	 {Proceedings of the 37th International Conference on Machine Learning},
  pages = 	 {11173--11182},
  year = 	 {2020},
    volume = 	 {119},
  series = 	 {Proceedings of Machine Learning Research},
    publisher =    {PMLR}
}

@Article{2S_Liu2021,
  author       = {Liu, An and Yang, Rui and Quek, Tony Q. S. and Zhao, Min-Jian},
  journal      = {IEEE Transactions on Signal Processing},
  title        = {Two-{Stage} {Stochastic} {Optimization} {Via} {Primal}-{Dual} {Decomposition} and {Deep} {Unrolling}},
  year         = {2021},
  issn         = {1941-0476},
  pages        = {3000--3015},
  volume       = {69},
  creationdate = {2024-04-24T15:44:17},
  doi          = {10.1109/TSP.2021.3079807},
  url          = {https://ieeexplore.ieee.org/document/9430669},
  urldate      = {2024-04-24},
}

@Article{2S_Zhai2022,
  author       = {Zhai, Xiongfei and Han, Guojun and Cai, Yunlong and Hanzo, Lajos},
  journal      = {IEEE Internet of Things Journal},
  title        = {Beamforming {Design} {Based} on {Two}-{Stage} {Stochastic} {Optimization} for {RIS}-{Assisted} {Over}-the-{Air} {Computation} {Systems}},
  year         = {2022},
  issn         = {2327-4662},
  month        = apr,
  number       = {7},
  pages        = {5474--5488},
  volume       = {9},
  creationdate = {2024-04-21T17:57:42},
  doi          = {10.1109/JIOT.2021.3108894},
  url          = {https://ieeexplore.ieee.org/document/9525086},
  urldate      = {2024-04-21},
}

@InProceedings{2S_Zhao2022,
  author       = {Zhao, Yangliu and Teng, Yinglei and Liu, An and Lau, Vincent},
  booktitle    = {{GLOBECOM} 2022 - 2022 {IEEE} {Global} {Communications} {Conference}},
  title        = {Two-{Timescale} {Joint} {UL}/{DL} {Dictionary} {Learning} and {Channel} {Estimation} in {Massive} {MIMO} {Systems}},
  year         = {2022},
  month        = dec,
  pages        = {5408--5413},
  creationdate = {2024-04-24T15:51:08},
  doi          = {10.1109/GLOBECOM48099.2022.10001679},
  url          = {https://ieeexplore.ieee.org/document/10001679/citations?tabFilter=papers#citations},
  urldate      = {2024-04-24},
}

@Article{2S_Zhao2024,
  author       = {Zhao, Yangliu and Teng, Yinglei and Liu, An and Wang, Xiaojuan and Lau, Vincent K. N.},
  journal      = {IEEE Transactions on Wireless Communications},
  title        = {Joint {UL}/{DL} {Dictionary} {Learning} and {Channel} {Estimation} via {Two}-{Timescale} {Optimization} in {Massive} {MIMO} {Systems}},
  year         = {2024},
  issn         = {1558-2248},
  month        = mar,
  number       = {3},
  pages        = {2369--2382},
  volume       = {23},
  creationdate = {2024-04-24T15:45:07},
  doi          = {10.1109/TWC.2023.3297995},
  url          = {https://ieeexplore.ieee.org/document/10197339},
  urldate      = {2024-04-24},
}

@INPROCEEDINGS{CV_10203333,
  author={Qin, Xiaorong and Song, Xinhang and Jiang, Shuqiang},
  booktitle={2023 IEEE/CVF Conference on Computer Vision and Pattern Recognition (CVPR)}, 
  title={Bi-Level Meta-Learning for Few-Shot Domain Generalization}, 
  year={2023},
  volume={},
  number={},
  pages={15900-15910},
  doi={10.1109/CVPR52729.2023.01526}}
\end{document}